\documentclass[12pt,reqno]{amsart}
\usepackage{amsmath,amsthm}
\usepackage[english]{babel}
\usepackage[margin=1.5in]{geometry}
\usepackage[applemac]{inputenc}
\usepackage{tikz}
\usetikzlibrary{arrows,automata}
\usepackage{amsfonts}
\usepackage{latexsym}
\usepackage{mathtools}
\usepackage{nccmath}
\usepackage{enumerate}
\usepackage[all,cmtip]{xy}
\usepackage{times}
\usepackage{tikz-cd}
\usepackage{nccmath}  

\newcommand{\al}{{\alpha}}

\newcommand{\T}{\tau} 
\newcommand{\C}{{\mathbb C}}
\newcommand{\J}{{\mathcal J}}
\newcommand{\F}{{\mathcal F}}
\newcommand{\G}{{\mathcal G}}

\newcommand{\hh}{{\mathcal H}}

\newcommand{\lb}{\lambda}

\newcommand{\pn}{{\mathbb P}^{n}}

\newcommand{{\pf}}{{\bf Proof. }}

\theoremstyle{plain}
\newtheorem{thm}{Theorem}[section]
\newtheorem{lem}[thm]{Lemma}
\newtheorem{cor}[thm]{Corollary}
\newtheorem{prop}[thm]{Proposition}

 \theoremstyle{definition}
\newtheorem{deff}[thm]{Definition}
\newtheorem{Conj}[thm]{Conjecture}
\newtheorem{Rem}[thm]{Remark}
\newtheorem{?}[thm]{Problem}
\newtheorem{Ex}[thm]{Example}

\makeatletter\@addtoreset{equation}{section} \makeatother

\usetikzlibrary{arrows}

\tikzset{
  treenode/.style = {align=center, inner sep=1pt, text centered,
    font=\sffamily},
  arn_n/.style = {treenode, circle, white, font=\sffamily\bfseries, draw=black,
    fill=black, text width=1.5em},
  arn_r/.style = {treenode, circle, black, draw=black,
    text width=1.5em, very thick},
  arn_x/.style = {treenode, rectangle, draw=black,
    minimum width=0.5em, minimum height=0.5em}
}

\begin{document}

\title{Self-similarity and Spectral Dynamics}
\author[B. Goldberg and R. Yang]{Bryan Goldberg and Rongwei Yang}
\address{State University of New York At Albany \\ Department of Mathematics \\
Albany NY 12222} 
\email{bgoldberg@albany.edu, ryang@albany.edu}

\begin{abstract} For a tuple $A= (A_0, A_1, \ldots , A_n)$ of elements in a unital Banach algebra $\mathcal{B}$, its  
\textit{projective (joint) spectrum} $p(A)$ is the collection of $z\in\mathbb{P}^{n}$ such that $A(z)=z_0A_0+z_1 A_1 + \ldots z_n A_n$ is 
not invertible. If the tuple $A$ is associated with the generators of a finitely generated group, then $p(A)$ is simply called the projective spectrum of the group. This paper investigates a connection between self-similar group representations and an induced polynomial map on the projective space that preserves the projective spectrum of the group. The focus is on two groups: the infinite dihedral group $D_\infty$ and the Grigorchuk group ${\mathcal G}$ of intermediate growth. The main theorem states that the Julia set of the induced rational map $F$ for $D_\infty$ is the union of the projective spectrum with the extended indeterminacy set. Moreover, the limit function of the iteration sequence $\{F^{\circ n}\}$ on the Fatou set is also fully described. The result has an application to the group ${\mathcal G}$ and gives rise to a conjecture about its associated Julia set.

\end{abstract}

\maketitle

\section*{Notations}

\begin{tabular}{ l l}
$\hat{\mathbb{C}}$ & $\mathbb{C} \cup \infty$, the Riemann sphere\\
$H^{\circ n}(z)$ & the $n^{th}$ iteration of a function $H(z)$\\
$\partial T$ & the boundary of the rooted binary tree\\
$\pi$ & the Koopman representation\\
$\lambda$ & the left regular representation\\
$E$ & the extended indeterminacy set\\
$\mathbb{P}^n$ & complex projective space of dimension $n$\\
$\phi$ & the cannonical projection from $\C^{n+1}$ to ${\mathbb P}^n$\\
$D_\infty$ &  $\langle a,t \mid a^2=t^2=1 \rangle$, the infinite dihedral group\\
$\mathcal{G}$ & the Grigorchuk group of intermediate growth\\
$\F(H)$ & the Fatou set of map $H$\\
$\J(H)$ & the Julia set of map $H$\\
$T(x)$ & $2x^2-1, x\in \hat{\mathbb{C}}$, the second Tchebyshev polynomial of the first kind
\end{tabular}

\section{Introduction} 

Let ${\mathcal B}$ be a complex unital Banach algebra and $A=(A_0, A_1, \cdots, A_n)$ be a tuple of linearly independent elements in ${\mathcal B}$. The multiparameter pencil \[A(z):=z_0A_0+z_1A_1+\cdots +z_nA_n\] is an important subject of study in numerous fields. The notion of projective (joint) spectrum is defined in \cite{Ya} as follows.

\begin{deff} For a tuple $(A_0, A_1, \cdots, A_n)$ of elements in a unital Banach algebra ${\mathcal B}$, its projective spectrum is defined as \[P(A) = \{z\in {\C}^{n+1}\mid A(z)\ \text{is not invertible}\}.\] The {\em projective resolvent set} refers to the complement, $P^c(A)={\C}^{n+1}\setminus P(A)$.
\end{deff}
Since $A(z)$ is homogeneous, projective spectrum and projective resolvent set can be defined in projective space as $p(A) = \{z\in {\mathbb{P}}^{n}\mid A(z)\ \text{is not invertible}\}$,
and $p^c(A)={\mathbb{P}}^{n}\setminus p(A)$. If $\phi$ denotes the cannonical projection from $\C^{n+1}$ to ${\mathbb P}^n$, then $p(A)=\phi(P(A))$ and $p^c(A)=\phi(P^c(A))$.
Projective spectrum provides an effective mechanism to study several operators simultaneously. It reveals joint behaviors of $A_0, A_1, \cdots, A_n$ as well as interactions among individual elements. Properties of projective spectrum have been previously investigated in a series of papers, such as \cite{BCY,CY,CSZ,DY,GOY,HWY,MQW17,SYZ} as well as some additional references therein. A particularly interesting case is when the tuple $A$ is associated with a finitely generated group. To be more precise, consider a group $G$ with a finite generating set $S=\{g_1,\ g_2,\ \dots,\ g_s\}$ and a unitary representation $\pi$ of $G$ on a complex Hilbert space ${\mathcal H}$, and let the multiparameter pencil be $A_{\pi}(z)=z_0I+z_1\pi(g_1)+\cdots +z_s\pi(g_s)$. Then the projective spectrum $P(A_\pi)$ encapsulates information about $G$ as well as the representation $\pi$. The most telling situation is when $\pi$ is finite dimensional, in which case the homogeneous polynomial
 \begin{equation*}
Q_{\pi}(z):=\det \big(z_0I+z_1\pi(g_1)+z_2\pi(g_2)+\cdots +z_s\pi(g_s)\big)
\end{equation*}
is called the characteristic polynomial of $G$ with respect to the representation $\pi$ and the generating set $S$. Clearly, in this case the projective spectrum $P(A_\pi)$ (or $p(A_\pi)$) is the zero set of $Q_{\pi}(z)$ in $\C^{n+1}$ (or $\pn$). In the case when $G=\{1, g_1, ..., g_n\}$ is a finite group and $\lambda$ is the left regular representation of $G$, the characteristic polynomial  $Q_{\lambda}(z)$ is called the group determinant of $G$, and its study can be traced back to Dedekind and Frobenius (\cite{Cu2, De, Fr}). It is indeed the birthplace of group representation theory. 
 Some further studies on group determinant can be found in \cite{Cu,Di21,Di75,FS}, and new studies related to projective spectrum of groups can be found in \cite{GY,HY18,ST,SYZ}. Projective spectrum can also detect some subtler information about $G$ and the representation $\pi$. 
For example, it is indicated in \cite{GY} with a simple proof that if $\pi$ and $\rho$ are two {\em weakly equivalent} unitary representations of $G$ then $P(A_\pi)=P(A_\rho)$. This fact will be used later in Section \ref{sec: WeakContainment}. 

Another important study of spectral theory on groups was done by Grigorchuk and his collaborators on the group $\mathcal G$ of intermediate growth. Some examples of this investigation include \cite{GN07,GNS,GNS2, GS}. Among many other results, they discovered that $\mathcal G$ has a {\em self-similar representation} on the rooted binary tree. More remarkably this self-similarity induces a rational map on a certain spectral set of $\mathcal G$ whose dynamical properties link tightly to spectral properties of ${\mathcal G}$.  This interaction between dynamics and spectral theory (called spectral dynamics for short) is the motivation for this paper, and we shall examine it from the viewpoint of projective spectrum.  The focus here is on the infinite dihedral group $D_{\infty}= \langle a, t \mid a^2=t^2=1 \rangle$ and the Grigorchuk group ${\mathcal G}$ of intermediate growth. For $D_\infty$ we consider the linear pencil $A_{\pi}(z)=z_0I+z_1\pi(a)+z_2\pi(t)$, where $\pi$ is the Koopman representation of $D_\infty$ on the rooted binary tree. Then the self-similarity of $\pi$ gives rise to a polynomial map $F: {\mathbb P}^2\to {\mathbb P}^2$. The following theorem determines the Julia set of $F$ (cf. Theorem \ref{thm: main}). \\

{\bf Main Theorem}: The Julia set of $F$ is the union of the projective spectrum $p(A_\pi)$ with the extended indeterminacy set of $F$.\\

This theorem and several other results about $D_\infty$ find applications to the Grigorchuk group ${\mathcal G}$, and they give rise to a conjecture about the ${\mathcal G}$'s associated Julia set. This paper is organized as follows.

\tableofcontents

\section{Projective spectrum of the dihedral group}

Although the dihedral group $D_\infty$ is probably the simplest non-abelian group, it has a wide range of applications, including pertinence to the Weyl group of simple Lie algebras as well as to the study of sophisticated groups such as the Grigorchuk group of intermediate growth (\cite{GY, Ne18}). The projective spectrum of $D_\infty$ with respect to the left regular representation $\lambda$ was computed and studied in depth in \cite{GY}. Here we give a different perspective of the spectrum through a spectral resolution of $\lambda$.

Consider the following two-dimensional representation $\rho_{\theta}$ of $D_{\infty}$ given by
\begin{equation}
\rho_{\theta}(a)=\begin{bmatrix}
         0 & e^{i \theta}\\
        e^{-i \theta} & 0
        \end{bmatrix},\ \ \ \rho_{\theta}(t)=\begin{bmatrix}
                                          0 & 1\\
                                          1 & 0
                                          \end{bmatrix},
\end{equation}
where $\theta\in [0, 2\pi)$. It is known from \cite{Hal,RS} that every irreducible unitary representation of $D_{\infty}$ is either one dimensional or of the form $\rho_{\theta}$ for some $\theta\in (0,\pi)$.
It is observed in \cite{GY} that the left regular representation of $D_{\infty}$ is equivalent to the following representation $\lambda$ of $D_{\infty}$ on $L^2({\mathbb T},\ \frac{d\theta}{2\pi})\oplus L^2({\mathbb T},\ \frac{d\theta}{2\pi})$ defined by:
\begin{equation}
\lambda  (a) = 
\begin{bmatrix}
0 & T \\
T^\ast & 0
\end{bmatrix}, \ \ \
\lambda (t) =
\begin{bmatrix}
0 & I_0 \\
I_0 & 0
\end{bmatrix},
\label{eq: lambda}
\end{equation}
where $I_0$ is the identity operator on $L^2({\mathbb T}, \frac{d \theta}{2 \pi})$, and 
$T: L^2({\mathbb T},\ \frac{d\theta}{2\pi})\to L^2({\mathbb T},\ \frac{d\theta}{2\pi})$
 is the bilateral shift operator, i.e., the unitary operator defined by $Tf(e^{i\theta})=e^{i\theta}f(e^{i\theta})$. If we let 
\begin{equation}
T = \displaystyle \int_0^{2\pi} e^{i \theta} dE(e^{i \theta})
\label{eq: T}
\end{equation}
be the spectral resolution of $T$, then bringing together (\ref{eq: lambda}) and (\ref{eq: T}), we have the following connection:
\begin{align*}
 \lambda  (a) &= 
\begin{bmatrix}
0 &  \displaystyle \int_0^{2\pi} e^{i \theta} dE(e^{i \theta}) \\
 \displaystyle \int_0^{2\pi} e^{-i \theta} dE(e^{i \theta}) & 0
\end{bmatrix}\\
&=\displaystyle \int_0 ^{2\pi}
\begin{bmatrix}
0 & e^{i \theta} \\
 e^{-i \theta} & 0
\end{bmatrix}
dE(e^{i\theta})=
\int_0 ^{2 \pi} \rho _\theta(a) dE(e^{i \theta});
\end{align*}
and
\begin{align*}
\lambda  (t) &= 
\begin{bmatrix}
0 &  \displaystyle \int_0^{2\pi}dE(e^{i \theta}) \\
 \displaystyle \int_0^{2\pi}  dE(e^{i \theta}) & 0
\end{bmatrix}\\
&=\displaystyle \int_0 ^{2\pi}
\begin{bmatrix}
0 &1 \\
1 & 0
\end{bmatrix}
dE(e^{i\theta})=\int_0 ^{2 \pi} \rho _\theta(t) dE(e^{i \theta}).
\end{align*}
In fact, the above observation gives rise to an explicit expression of the regular representation of $D_\infty$ in terms of a direct integral of irreducible representations. Now we consider the linear pencil $A_{\lb}(z)=z_0I+z_1\lb(a)+z_2\lb(t)$, where $z~=~(z_0, z_1, z_2)~\in~{\mathbb C}^3$. The above observation yields
\begin{equation}\label{eq: atheta}
A_\lb(z)=\int_0^{2\pi} A_{\rho_\theta}(z)dE(e^{i\theta}).
\end{equation}
Hence, the pencil $A_{\lb}(z)$ is not invertible if and only if $A_{\rho_\theta}(z)$ is not invertible for at least one $\theta\in [0, 2\pi)$. Thus we obtain the following theorem which was originally proved in \cite[Theorem 1.1]{GY}.
\begin{thm}
\label{thm: PA}
Let $\lambda : D_\infty\to l^2(D_\infty)$ be the left regular representation. Then
\[
P(A_{\lb})=\bigcup_{0\leq \theta < 2\pi}\{z\in \C^3 \mid  z_0^2-z_1^2-z_2^2-2z_1z_2 \cos \theta=0\}.
\]
\end{thm}
In the projective space, ${\mathbb P}^2$ we have an equivalent expression
\[p(A_{\lb})=\bigcup_{0\leq \theta < 2\pi}\{z\in \mathbb{P}^2 \mid  z_0^2-z_1^2-z_2^2-2z_1z_2 \cos \theta=0\}.\]
Theorem \ref{thm: PA}, in particular its equivalent statement in ${\mathbb P}^2$, is fundamental to this paper. Observe that the theorem is equivalent to stating that if $z \in P(A_\lb)$ then 
\begin{equation}
L_x(z):=z_0^2-z_1^2-z_2^2-2z_1 z_2 x = 0 \text{ for some } x \in [-1,1].
\end{equation}
For simplicity, we denote the quadratic surface 
\begin{equation}
S_x = \{ z \in \mathbb{C}^3 \mid L_x(z) = 0\}.
\end{equation} 
Observe that the union in Theorem \ref{thm: PA} is not a disjoint union of quadratic surfaces. For instance, the point $(1, 1, 0)\in S_x$ for every $x\in [-1, 1]$. If $z \in P(A_\lb)$ and $z_1z_2 \neq 0$ then there is a unique $\hat{x} \in [-1,1]$ such that $z \in S_{\hat{x}}$, e.g., the function $\tau(z)=\frac{z_0^2-z_1^2-z_2^2}{2z_1z_2}$ is well-defined (in $ \C^3$ as well as in  ${\mathbb P}^2$). The function $\tau$ will be more rigorously defined later and it is instrumental to the study of spectral dynamics of $D_\infty$.

\section{Self-similar representation}

Consider a countable group $G$ and a unitary representation $\rho: G\to U({\mathcal H})$, where $U({\mathcal H})$ denotes the group of unitary operators on the Hilbert space ${\mathcal H}$. We shall call the representation self-similar, or more precisely $d$-similar, if there exists a natural number $d$ and a unitary map $W: \hh\to \hh^d$ such that for every $g\in G$ the $d\times d$ block matrix $X(g):=W\rho(g)W^*$ has all of its entries $x_{ij}$ either equal to $0$ or from $\rho(G)$. Observe that in this case, since $X(g)$ itself and each of its nonzero entry $x_{ij}$ are unitaries, every row or column of $X(g)$ has precisely one nonzero entry. For more details on self-similarity, we refer the readers to \cite{GNS, GY, Ne05}. Self-similar representations sometimes arise when the group $G$ acts on a rooted tree. Figure \ref{eq: tree} below shows the first three levels of an infinite binary tree. 

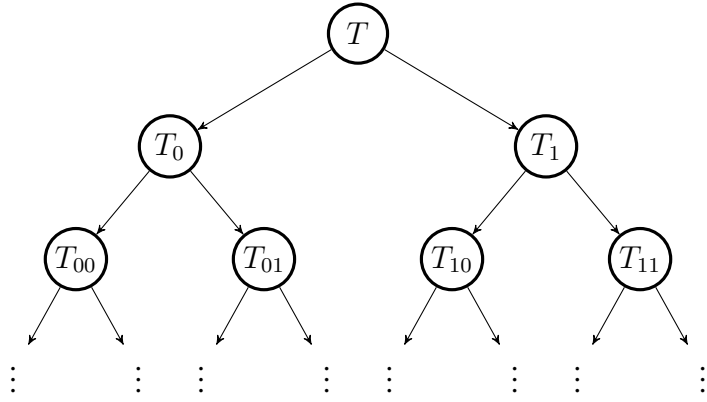
\begin{figure}[h!]
\begin{tikzpicture}[->,>=stealth',level/.style={sibling distance = 5cm/#1,
  level distance = 1.5cm}]
\node [arn_r] {$T$}
    child{ node [arn_r] {$T_0$}
            child{ node [arn_r] {$T_{00}$}
                    child{ node {\vdots}}
                    child{ node {\vdots}}}
            child{ node [arn_r] {$T_{01}$}
                    child{ node {\vdots}}
                    child{ node {\vdots}}}
           }
    child{ node [arn_r] {$T_1$}
            child{ node [arn_r] {$T_{10}$}
                     child{ node {\vdots}}
                     child{ node {\vdots}}}
            child{ node [arn_r] {$T_{11}$}
                      child{ node {\vdots}}
                      child{ node {\vdots}}}
	   }
;
\end{tikzpicture}
\caption{Rooted binary tree}
\label{eq: tree}
\end{figure}
Clearly, the tree $T$ consists of two subtrees $T_0$ and $T_1$, each of which also consists of two subtrees: $T_{00}$ and $T_{01}$, and $T_{10}$ and $T_{11}$, respectively, etc. The boundary $\partial T$ of the tree $T$ is the collection of all infinite sequences of directed arrows from the vertex $T$ down the tree. The uniform Bernoulli measure $\mu$ on $\partial T$ is defined by $\mu(\partial T_{i_1\cdots i_p})=\frac{1}{2^p}$, where $i_k\in \{0, 1\}$ for each $k$. In other words, the measure $\mu$ distributes evenly on the subtrees  at every level. Since every element in $\partial T$ corresponds to an infinite sequence of directed arrows in $T$, it thus corresponds to a unique sequence of $0$s and $1$s. Hence there is a natural bijection from $\partial T$ to the interval $[0, 1]$ (expressed in binary numbers). And this bijection also identifies the measure $\mu$ with the Lebesgue measure on $[0, 1]$.

Define the Hilbert space $\hh=L^2(\partial T,\mu)$. Let $\mu_i=2\mu,\ i=0, 1$ be the normalized restrictions of $\mu$ on the boundary of the subtrees $ \partial T_0,\ \partial T_1$, and define $\hh_i=L^2(\partial T_i,\ \mu_i),$ $i=0, 1$. Then each $\hh_i$ can be identified with $\hh$ and hence $\hh=\hh_0\oplus \hh_1$ can be identified with $\hh\oplus \hh$ by a unitary $W$. 

If a locally compact group $G$ has a measure-preserving action on a measure space $(X, \mu)$, then the Koopman representation $\pi: G\to U\big(L^2(X, \mu)\big)$ is defined by 
\[\pi(g)f(x)=f(g^{-1}x),\ \ \forall x\in X,\ \forall g\in G.\]
Clearly, if $X=G$ and $\mu$ is the Haar measure, then in this case the Koopman representation is the left regular representation of $G$. This paper is primarily concerned with the Koopman representation of two groups: the infinite dihedral group and the Grigorchuk group of intermediate growth, both of which have a measure preserving action on the rooted binary tree $T$.

\subsection{Weak containment and projective spectrum}
\label{sec: WeakContainment}
\deff Consider two unitary representations $\pi$ and $\rho$ of a discrete group $G$ in Hilbert spaces ${\mathcal H}$ and ${\mathcal K}$, respectively. One says that $\pi$ is weakly contained in $\rho$ (denoted by $\pi\prec \rho$) if for every $x\in {\mathcal H}$, every finite subset $F\subset G$, and every $\epsilon>0$ there exist $y_1,\ y_2,\ \cdots,\ y_n$ in ${\mathcal K}$ such that for all $g\in F$
\[\big|\langle \pi(g)x,\ x\rangle-\sum_{i=1}^{n}\langle \rho(g)y_i,\ y_i\rangle \big|<\epsilon.\]

Weak containment plays an important role in the representation theory of countable groups. Several equivalent conditions are discussed in \cite{BHV, Di}. Two unitary representations $\pi$ and $\rho$ are said to be weakly equivalent if $\pi\prec \rho$ and $\rho\prec \pi$. It was recognized in \cite{GY} that projective spectrum is invariant with respect to the weak equivalence of representations. 

\begin{prop}\label{prop: weak}
If $\pi $ and $\rho$ are two weakly equivalent unitary representations of the group $G=\langle g_1, g_2, ..., g_n \rangle$, then $P(A_\pi)=P(A_\rho)$.
\end{prop}

In particular, the following theorem was proved in the case of $D_\infty$  \cite[Theorem 7.2]{GY}. 

\begin{thm}\label{thm: equ}
Let $\lambda$ be the left regular representation of $D_\infty$ and $\pi$ be the Koopman representation of $D_\infty$ on the binary tree $T$. Then $\lambda$ and $\pi$ are weakly equivalent. In particular, we have $P(A_\lb)=P(A_\pi)$.
\end{thm}

\subsection{The dihedral group and dynamical map}
\label{sec: Dinf}

The Koopman representation of $D_{\infty}$ on the tree $T$ is realized by the following automaton in Figure 2 
\begin{figure}[h]
\begin{tikzpicture}[>=stealth',shorten >=1pt,auto,node distance=3cm]
 \node[state,label=below:$t$] (S) {$e$};
\node[state,label=below:$a$]         (T) [right of=S] {$\sigma$};
\node[state,label=below:$id$]         (U) [right of=T] {$e$};
\path[->] (S)  edge [loop above] node {1} (S);
 \path[->]            (S) edge              node {0} (T);
\path[->] (T)       edge              node {0,1} (U);
\path[->] (U) edge [loop above] node {0,1} (U);
\end{tikzpicture}
\caption{Automaton of the group $D_\infty$}
\label{fig: Aut}
\end{figure}
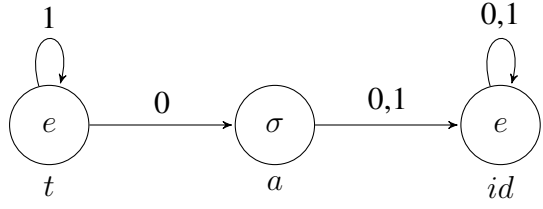
where $a$ and $t$ are automorphisms of $T$ satisfying the recursive relation
\begin{align}
a=\sigma,\ \ \ \ t=(a,\ t).
\end{align}
Here $\sigma$ stands for the involution that exchanges the subtrees $T_0$ and $T_1$, and $(a, t)$ stands for the diagonal block matrix $a\oplus t$. This automaton indicates that the Koopman representation of $D_\infty$ on $\hh=L^2(\partial T, \mu)$ is $2$-similar (\cite{GY}), and the identification $W:\hh=\hh_0\oplus \hh_1\to \hh\oplus \hh$ mentioned earlier gives rise to the unitary equivalence
\[
W\pi (a)W^* = 
\begin{bmatrix}
0 & I \\
I & 0
\end{bmatrix}, \ \ \
W\pi (t)W^* =
\begin{bmatrix}
\pi (a)  & 0 \\
0 & \rho (t)
\end{bmatrix}.\]
For simplicity we shall use the symbol ``$\cong$" to denote the above unitary equivalence. Thus we have
\begin{equation}\label{block}
A_{\pi}(z) = z_0 +z_1 \pi (a) +z_2 \pi (t) \cong   \begin{bmatrix}
z_0+z_2 \pi(a) &	 z_1 \\
z_1  &	 z_0+z_2\pi(t)
\end{bmatrix}.
\end{equation}
Therefore $A_{\pi}(z)$ is invertible if and only if the $2\times 2$ block matrix on the right side of (\ref{block}) is invertible.
For convenience, here we shall write $\pi(a)$ and $\pi(t)$ simply as $a$ and $t$, respectively.
If $z_0^2 \neq z_2^2$, then $z_0+z_2a$ is invertible and its inverse is $(z_0-z_2 a)(z_0^2-z_2 ^2)^{-1}$. By a Schur complement argument, the block matrix $A_\pi(z)$ is invertible if and only if  $ z_0+z_2t-z_1 ^2(z_0-z_2a)(z_0^2-z_2^2)^{-1}$ is invertible, or if and only if the rational pencil
\[ \displaystyle \frac{z_0(z_0^2-z_1^2-z_2^2)}{z_0^2-z_2^2}+\frac{z_1 ^2z_2}{z_0^2-z_2^2}a +  z_2 t\] 
is invertible. This leads one to define the following polynomial map (\cite{GY}):
\begin{equation}
\label{eq: F}
F_1(z_0,z_1,z_2) = \bigg( z_0(z_0^2-z_1^2-z_2^2),  z_1^2 z_2,  z_2(z_0^2-z_2^2) \bigg). 
\end{equation}  

\Rem
The symmetry of $a$ and $t$ in $D_\infty$ implies the symmetry of $P(A_\pi)$ over $z_1$ and $z_2$. This fact enables one to define the symmetric map
\begin{equation}
\label{eq: F2}
F_2(z) = \big(z_0(z_0^2-z_1^2-z_2^2), z_1(z_0^2-z_1^2) , z_1 z_2^2 \big).
\end{equation}
Since the results will be parallel, we shall focus on $F_1(z)$ and simply write it as $F(z)$.
It was proved in \cite[Theorem 8.3]{GY} that the projective spectrum $P(A_\pi)$ is invariant under $F$. Thus, in view of Theorem \ref{thm: equ} we have
\begin{prop}\label{thm: Inv}
Let $\pi$ be the Koopman representation of $D_\infty$ on the rooted binary tree. Then the projective spectrum $P(A_\pi)$ is invariant under $F$.
\end{prop}
This proposition is another important motivation for our investigation on the dynamics of $F$. We now examine a component of its proof in \cite{GY} more closely. Suppose $z\in \C^3$ and $z'=F(z)$. Suppose also that $z_1 z_2\neq 0$ and $z'_1 z'_2\neq 0$. Then
there exist unique $x$ and $x'$ in $\C$ such that \[z_0^2-z_1^2-z_2^2-2z_1z_2x=0\ \ \text{and}\ \ z_0'^2-z_1'^2-z_2'^2-2z_1'z_2'x'=0.\]
In particular, Proposition \ref{thm: Inv} shows that if $z\in P(A_\pi)$ then $z'\in P(A_\pi)$, and hence by Theorem \ref{thm: PA} both $x$ and $x'$ are in the interval $[-1, 1]$. To elucidate the connection between $x$ and $x'$, one computes that
\begin{align*}
x' = &\frac{z_0'^2-z_1'^2-z_2'^2}{2z_1'z_2'}\\
=& \frac{z_0^4+z_1^4+z_2^4-2z_0^2z_1^2-2z_0^2z_2^2}{2z_1^2z_2^2}\\
=& 2\bigg(\frac{z_0^2-z_1^2-z_2^2}{2z_1 z_2}\bigg)^2 -1=2x^2-1.
\end{align*}
The polynomial $T(x)=2x^2-1$ is the second Tchebyshev polynomial of the first kind. To further explore this connection, we use the Riemann sphere $\hat{\C}$ and define the function $\hat{\tau}: \C^3\to \hat{\C}$ by 
\begin{equation}
\label{eq: tauhat}
\hat{\tau}(z) = \begin{dcases}
      0 & \text{ if } z\in S_0=\{z \in \mathbb{C}^3 \mid z_0^2-z_1^2-z_2^2=0\}, \\
      \dfrac{z_0^2-z_1^2-z_2^2}{2z_1z_2} &  \text{ otherwise. }
    \end{dcases}
\end{equation}
By construction $\hat{\tau}$ is a well-defined map from ${\mathbb C}^3$ into $\hat{\C}$, and it is holomorphic on ${\mathbb C}^3\setminus S_0$. It may not be continuous at some points in $S_0$. For instance, for any $n\in {\mathbb N}$ one has $\hat{\tau} (1+\frac{1}{n}, 0, 1)=\infty$ but $\hat{\tau} (1, 0, 1)=0$ by definition. But this discontinuity will not affect later discussions. Equipped with the function $\hat{\tau}$, the relation between $x$ and $x'$ can be re-stated as 
\begin{equation}
\label{eq: Tche} 
\hat{\tau}(F(z))=T(\hat{\tau}(z)), \ \forall z\in \C^3.
\end{equation}
It is entertaining to check that Equation (\ref{eq: Tche}) also holds for the case $z_1z_2=0$. Equivalently one has the following commutative diagram
\begin{displaymath}
\label{diag: tauhat}
  \xymatrix{
    \mathbb{C}^3 \ar[r]^{F}  \ar[d]^{\hat{\tau}}
    &  \mathbb{C}^3 \ar[d]^{\hat{\tau}} \\
    {\hat{\mathbb{C}}} \ar[r]^{T} 
    & {\hat{\mathbb{C}}}.
  }
\end{displaymath}
This relationship, often called {\em semi-conjugacy}, betwen the map $F$ and the Tchebyshev polynomial $T$ is instrumental to our investigation of the dynamics of $F$. We will address more on this in Section \ref{sec: dynamics}.

\subsection{The Grigorchuk group and dynamical map}

The Grigorchuk group $\G$ is generated by four involutions $a, b, c, d$ with the following infinite set of algebraic equations:
\begin{align*}
 & a^2=b^2=c^2=d^2=bcd=1,\\
& \sigma^k((ad)^4)=\sigma^k((adacac)^4)=1,\ k=0,\ 1,\ 2,\ \cdots, 
\end{align*}
where $\sigma$ is the substitution:$\ a\to aca,\ b\to d,\ c\to b,\ d\to c$. It is the first example of group of intermediate growth (between polynomial growth and exponential growth) which settled a problem posed by J. Milnor (cf. \cite{Gr80}). The Koopman representation $\pi$ of ${\mathcal G}$ on the binary tree $T$ is realized by the following $5$-state automaton: 
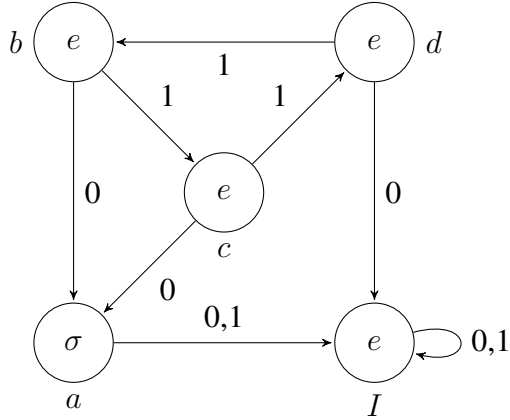
\begin{figure}[h]
\begin{tikzpicture}[>=stealth',shorten >=1pt,auto,node distance=3cm]
  \node[state,label=left:$b$] (A) at ( -2,2) [shape=circle,draw]{$e$};
  \node[state,label=right:$d$] (B) at ( 2,2) [shape=circle,draw]{$e$};
  \node[state,label=below:$c$] (C) at ( 0,0) [shape=circle,draw]{$e$};
  \node[state,label=below:$a$] (D) at ( -2,-2) [shape=circle,draw] {$\sigma$};
  \node[state,label=below:$I$] (E) at (2,-2) [shape=circle,draw]{$e$};
\path[->]
        (B)   edge                    node {1}            (A)
        (A)   edge                    node {0}            (D)
        (A)   edge                    node {1}            (C)
        (C)   edge                    node {0}            (D)
        (C)   edge                    node {1}            (B)
        (D)   edge                    node {0,1}         (E)
        (B)   edge                    node {0}            (E)
        (E)   edge  [loop right]   node {0,1}            (E);
\end{tikzpicture}
\caption{Automaton of the group ${\mathcal G}$}
\end{figure}

The invertibility of the linear pencil \[Q(x_1, x_2)=-x_1 \pi(a)+\pi(b)+\pi(c)+\pi(d)-(x_2+1)I,\] where $x_1, x_2\in \mathbb R$, has been studied (cf. \cite{BG,Gr84,Gr85}). Among other things, the following two important discoveries were made.
\begin{enumerate}
\item The Koopman representation $\pi$ of ${\mathcal G}$ is $2$-similar in the sense that there is a unitary map $W: {\mathcal H}\to {\mathcal H}^2$ such that  \begin{align*}
W\pi(a)W^*&=\left(\begin{matrix}
                      0 & I\\
                      I & 0
                     \end{matrix}\right),\hspace{15mm} W\pi(b)W^*=\left(\begin{matrix}
                                                                      \pi(a) & 0\\
                                                                      0 & \pi(c)
                                                                     \end{matrix}\right),\\
W\pi(c)W^*&=\left(\begin{matrix}
                      \pi(a) & 0\\
                      0 & \pi(d)
                     \end{matrix}\right),\ \ \ W\pi(d)W^*=\left(\begin{matrix}
                                                                      I & 0\\
                                                                      0 & \pi(b)
                                                                     \end{matrix}\right).
\end{align*}

\item The above self-similarity is reflected by a real rational map on the spectral set \[P(Q)=\{(x_1,x_2)\in \mathbb R^2 \mid Q(x_1,x_2)\ \text{is not invertible}\}.\]
\end{enumerate}
Here, in place of the real pencil $Q(x_1,x_2)$ we shall consider the homogeneous complex pencil
\[R_\pi(z) = z_0I +z_1\pi(a)+z_2\pi(b)+z_3\pi(c)+z_4\pi(d),\ \ z\in \C^5\] and the associated projective spectrum $P(R_\pi)$. Again, for convenience we shall write $\pi(a)$ simply as $a$, etc., in the sequel. Then the above $2$-similarity implies
\begin{equation}
R_\pi(z)\cong \begin{bmatrix}
z_0+z_2a+z_3a+z_4I & z_1 \\
z_1 & z_0+z_2c+z_3d+z_4b
\end{bmatrix}. 
\label{eq: block}
\end{equation}
If $(z_0+z_4)^2 \neq (z_2+z_3)^2$, then $(z_0+z_4)+(z_2+z_3)a$ is invertible, and its inverse is
\[\frac{(z_0+z_4)-(z_2+z_3)a}{(z_0+z_4)^2-(z_2+z_3)^2}, \] and hence by a Schur complement argument $R_\pi(z)$ is invertible if and only if: 
\begin{equation}
\label{eq: G schur's}
z_0+z_2c+z_3d+z_4b-z_1^2[z_0+z_4 +(z_2+z_3)a]^{-1}
\end{equation}
 is invertible. 
Re-writing (\ref{eq: G schur's}) as
\[ z_0+z_2c +z_3d+z_4b - \frac{z_1^2}{(z_0+z_4)^2 - (z_2+z_3)^2}\big(z_0+z_4-(z_2+z_3)a\big)\]
and then re-grouping the terms by the generators, one discovers that $R_\pi(z)$ is invertible if and only if the rational pencil
\begin{equation}
 \big(z_0 - \frac{z_1^2(z_0+z_4)}{(z_0+z_4)^2 - (z_2+z_3)^2}\big) I + \frac{z_1^2(z_2+z_3)}{(z_0+z_4)^2 - (z_2+z_3)^2} a +z_4b +z_2c+z_3d
\end{equation}
is invertible. Note, this fact was first observed for the case $z_2=z_3=z_4$ in \cite{BG}. 
 In order to obtain a polynomial map we multiply the above rational pencil by the function $\al(z) = (z_0+z_4)^2 - (z_2+z_3)^2$ and this led us to define the map $G: \C^5 \to \C^5$ as
\begin{equation}
\label{eq: G(z)}
 G(z) = \big(z_0 \al-z_1^2(z_0+z_4), z_1^2 (z_2+z_3), z_4 \al, z_2 \al, z_3 \al \big).
\end{equation}

 We begin by demonstrating two initial observations about the function $\alpha$'s zero set  $J=\{z\in{\C^5} | \al(z)=0\}$. 
\begin{prop}\label{thm: J}
(a) $G(J)\subset P(R_\pi); \ \ (b)\ G^{\circ 2}(J)=\{{\bf 0}=(0,0,0,0,0)\}$.
\end{prop}
\begin{proof}
First, for part (a) since the group element $a$ is an involution, we have the classical spectrum $\sigma(\pi(a))=\{\pm 1\}$. For each $z\in J$ we have $(z_0+z_4)=\pm (z_2+z_3)$, and hence in this case $G(z)$ can be written as $z_1^2(z_2+z_3)\big(\pm 1, 1, 0, 0, 0\big)$ which is in $P(R_\pi)$.
Moreover, since $G$ is homogeneous and $G(\pm 1, 1, 0, 0, 0)={\bf 0}$, part (b) follows.
\end{proof}

The above observations lead to the following lemma which motivated our study of the dynamics of $G$.
\begin{prop}
Consider the Grigorchuk group $\mathcal G$ and its Koopman representation $\pi$ on the rooted binary tree. For the map $G:\C^5\to \C^5$ defined in (\ref{eq: G(z)}), we have 

(a) $G(P(R_\pi))\subset P(R_\pi)$;

(b) $G(P^c(R_\pi)\setminus J)\subset P^c(R_\pi)$.
\end{prop}
\begin{proof}
 If $z\notin J$, then by the self-similarity of the representation $\pi$ the multiparameter pencil $R_{\pi}(z)$ is invertible if and only if the block matrix (\ref{eq: block}) is invertible. Then by a Schur complement argument it is invertible if and only if the pencil (3.9) is invertible. The proposition thus follows from Proposition \ref{thm: J} (b).
\end{proof}

\section{Some fundamentals of complex dynamics}
\label{sec: dynamics}
For a function $R: \hat{\C}\to \hat{\C}$, we shall use the notation $R^{\circ n}$ to denote the $n$-th iteration of $R$. For instance, the notation $R^{\circ 3}(z)$ stands for $R(R(R(z)))$. Complex dynamics studies various issues concerning the convergence of the sequence $\{R^{\circ n}\}_{n=1}^{\infty}$. The readers shall find more information on this subject in \cite{Be, JM}. Here we only mention some definitions and facts pertaining to our study.

\deff Given a non-constant rational function $R: \hat{\C}\to \hat{\C}$. Its {\em Fatou set} ${\mathcal F}(R)$ is the maximal open subset of $\hat{\C}$ on which the sequence $\{R^{\circ n}\}_{n=1}^{\infty}$ is
equicontinuous. The Julia set ${\mathcal J}(R)$ is the complement $\hat{\C}\setminus {\mathcal F}(R)$.\\

It is known that there exists rational maps for which the Fatou set is empty (\cite{La}). For the Tchebyshev polynomial mentioned earlier, the following theorem is known (\cite{Be}). 

\begin{thm}
\label{thm: Tch}
Consider the map $T(x)=2x^2-1$ on $\hat{\C}$. Then its Julia set ${\mathcal J}(T)=[-1, 1]$. Further, the iteration sequence \{$T^{\circ n}\}$ converges normally to $\infty$ on the Fatou set $\hat{\C}\setminus [-1, 1]$.
\end{thm}

\subsection{Dynamics in ${\C}^{n+1}$}

For a point $z=(z_0, z_1, ..., z_{n})\in {\C}^{n+1}$ and any $p>0$, its $p$-norm $\|\cdot\|_p$ is defined by $\|z\|_p=(|z_0|^p+\cdots +|z_{n}|^p)^{1/p}.$
Now consider a polynomial map $H=(H_0, ..., H_n): {\C}^{n+1}\to {\C}^{n+1}$. Its Jacobian matrix is the $(n+1)\times (n+1)$ matrix \[H'(z):=\left(\frac{\partial H_i}{\partial z_j}\right)_{i,j=0}^n.\] 
A point $z\in {\C}^{n+1}$ is called a fixed point of $H$ if $H(z)=z$. Using the eigenvalues of the Jacobian matrix one can classify fixed points by their behavior such as \textit{attracting, repelling, parabolic}, or \textit{saddle points} (\cite{Forn}) which we will define later.  Suppose $z$ is an attracting fixed point for $H$, then the {\em basin of attraction} around $z$ is the maximal open domain $B_z\subset  {\C}^{n+1}$ containing $z$ such that for every $w\in B_z$ one has $H^{\circ n}(w)\to z$. 
Clearly, not every map $H$ has a fixed point or a basin of attraction. For example, consider the Tchebyshev polynomial $T(x)=2x^2-1$. Theorem \ref{thm: Tch} indicates that, if viewed as a map from $\C\to \C$ then $T$ has no basin of attraction; but if $T$ is regarded as a map from $\hat{\C}\to \hat{\C}$ then $\infty$ is a fixed point and $\hat{\C}\setminus [-1, 1]$ is the basin of attraction of $\infty$.

Regarding the maps $F$ and $G$, one has the following simple observation. 

\begin{prop}\label{eq: basin}
For the map $F: \C^3\to \C^3$ defined in (\ref{eq: F}) and the map $G:\C^5\to \C^5$ defined in (\ref{eq: G(z)}) we have
\begin{enumerate}[(a)]
\item $\|F(z)\|_2\leq \|z\|^3_2$; 

\item $\|G(z)\|_1\leq \|z\|^3_1$.
\end{enumerate}
\end{prop}
\begin{proof}
 For (a), one verifies that
\begin{align*}
\|F(z)\|_2^2&=|z_0(z_0^2-z_1^2-z_2^2)|^2+|z_1^2z_2|^2+|z_2(z_0^2-z_2^2)|^2\\
&\leq |z_0|^2\|z\|_2^4+|z_2|^2(|z_1|^4+\big(|z_0|^2+|z_2|^2)^2\big)\\
&\leq (|z_0|^2+|z_2|^2)\|z\|_2^4\leq \|z\|_2^6.
\end{align*}

For (b), one checks that
\begin{align*}
\|G(z)\|_1&=|z_0\al-z_1^2(z_0+z_4)|+|z_1^2(z_2+z_3)|+|\al|(|z_2|+|z_3|+|z_4|)\\
&\leq |\al|((|z_0|+|z_2|+|z_3|+|z_4|)+|z_1|^2(|z_0+z_4|+|z_2+z_3|)\\
&\leq (|z_1|^2+|z_0+z_4|^2+|z_2+z_3|^2)\|z\|_1\leq \|z\|_1^3.
\end{align*}
\end{proof}
It is clear from Proposition \ref{eq: basin} that the unit balls $\{z\in \C^3 \mid \|z\|_2<1\}$ and $\{z\in \C^5 \mid \|z\|_1<1\}$ are inside the basins of attraction for the maps $F$ and $G$, respectively. However, it seems difficult to completely determine the basin of attraction for either $F$ or $G$.

Fixed points of a map $H$ are classified according to the moduli of the eigenvalues of its Jacobian matrix $H'(z)$ ( \cite{Forn, MNTU}). A fixed point $z$ is said to be

\begin{itemize}
\item {\it attracting} if the moduli of all the eigenvalues of $H'(z)$ are strictly less than 1.
\item {\it super-attracting} if all the eigenvalues are zero.
\item {\it repelling} if the moduli of all the eigenvalues of $H'(z)$ are strictly greater than 1.
\item {\it parabolic} if at least one eigenvalue has norm exactly one.
\item {\it a saddle point} if some eigenvalues are strictly less than $1$, and other eigenvalues are strictly greater than $1$. 
\end{itemize}

Determining the fixed points of $H$ and their types requires only direct computation. Hence we leave the proofs for the following results in the Appendix.

\begin{prop}\label{prop: Ffix}
Consider the infinite dihedral group $D_\infty$ and the associated map $F$ defined in (\ref{eq: F}).

(a) Regarded as a map from $\C^3\to \C^3$,  the points $(0,0,0), (0, \pm i, \mp i)$  and the points in the surface \[Y_F:=\{ z~\in~\mathbb{C}^3 \mid z_1 =0, z_0 ^2 -z_2 ^2 =1\}\] are fixed under $F$.

(b) The origin $(0, 0, 0)$ is the only attracting fixed point.

(c) Regarded as a map from $\mathbb{P}^2\to \mathbb{P}^2$, the set of fixed points of $F$ is 
\[\{z\in \mathbb{P}^2\mid z_1=0, z_0\neq \pm z_2\}\cup \{[0: 1: -1]\}.\]
\end{prop}

The fixed points for the map $G$ defined in (\ref{eq: G(z)}) are more complicated. To conveniently describe them, we set 
 \[Y_1:=\{z \in \C^5 \mid z= (-\gamma \pm \sqrt{1+4 \gamma^2},0, \gamma, \gamma, \gamma), \gamma \neq 0\}.\]
Further, we let 
\[w_{\pm}(\zeta)= \frac{-4 \zeta -3 \pm i \sqrt{8\zeta +3}}{2 \zeta +10},\ \ \zeta\in \C,\] and define
\[Y_2:=\bigg\{ \left(\frac{\beta}{\beta^2(\alpha-1)-1}, \frac{-1}{\beta}, \alpha \beta, \alpha ^2 \beta, \beta \right), \alpha^3=1, \alpha\neq 1, \beta^2 =w_{\pm}(\alpha) \bigg\}.\]

\begin{prop}\label{prop: Gfix}
Consider the Grigorchuk group $\mathcal{G}$ and the associated map $G$ defined in (\ref{eq: G(z)}).

(a) Regarded as a map from $\C^5\to \C^5$, the set of fixed points of $G$ is 
\[Y_1\cup Y_2\cup \{{\bf 0}\}.\]

(b) The origin ${\bf 0}$ is the only attracting fixed point of $G$.

(c) Regarded as a map from $\mathbb{P}^4\to \mathbb{P}^4$, the set of fixed points of $G$ is
\[\phi(Y_1)\cup \phi(Y_2)\cup \{[-1: -2: 1: 1: 1]\},\]
where $\phi: \C^5\to {\mathbb P}^4$ is the cannonical projection.
\end{prop}
We will verify in Section 6.1 that the point $[-1: -2: 1:1:1]$ is in $p(R_\pi)$ while the set $\phi(Y_1)$ is contained in $p^c(R_\pi)$. However, situation for the points in $Y_2$ is not clear at this point.

Proposition \ref{prop: Ffix} (b) and Proposition \ref{prop: Gfix} (b) make one wonder if the origin is the only attracting fixed point for other self-similar groups. 


\subsection{Dynamics in complex projective space}

Consider a polynomial map $H: {\mathbb P}^2\to {\mathbb P}^2$ defined by $H(z)=[A(z):B(z):C(z)]$,
where $A, B$, and $C$ are homogeneous polynomials in $[z_0: z_1: z_2]$ of the same degree $d\geq 2$. Be aware that this map $H$ is not well-defined at the common zeros of $A, B,$ and $C$ because ${\mathbb P}^2$ contains no origin. In the case there are only a finite number of such common zeros, the map $H$ is called a {\em rational map}. The possible presence of common zeros of $A, B,$ and $C$ adds quite a bit more complication to the study of dynamics on ${\mathbb P}^2$. 

\deff If $H= [A:B:C]$ where $A, B, C$ are homogenuous polynomials of the same degree $d\geq 2$ then an {\em indeterminacy point} $z\in {\mathbb P}^2$ is such that $A(z) =B(z)=C(z)=0$. The set of indeterminacy points of $H$ will be denoted as $I(H)$, or simply $I$ when there is no confusion.\\

In order to study the iterations $H^{\circ n}$ one must consider their indeterminacy points as well (\cite{FS2}).

\deff An orbit $\displaystyle \{p_n \}_{n=-k} ^0$ is said to be complete if
\begin{enumerate}
\item $H(p_n) = p_{n+1}$,

\item $p_0 \in I$,

\item $p_n \notin I$ \quad if $n < 0 $,

\item if $k$ is finite, $p_{-k} \notin H(\mathbb{P}^2 \setminus I)$.
\end{enumerate}
We call $k+1$ the length of the orbit.

\deff A point $p \in \mathbb{P}^2$ is a point of indeterminacy for $H^{\circ n}$ if
\[\{ p, H(p), \ldots ,H^{\circ ( k-1)} (p) \}\] is a right tail of some complete orbit for some $1 \leq k \leq n$.
The set of such points is denoted by $I_n$.\\

The following fact is not hard to see.
\begin{cor}
If $I_n$ denotes the indeterminacy set of $H^{\circ n}$, then $I_n \subset I_m$, $\forall m > n$. 
\end{cor}

\deff Let $I_\infty=\cup_{n=1}^{\infty}I_n$. The extended indeterminacy set of $H$ is defined as $E_H=\overline{I_\infty}$, where the closure is with respect to the Fubini-Study metric on ${\mathbb P}^2$.\\

\noindent We shall write $E_H$ simply as $E$ when there is no confusion about the map $H$. As in the one variable case, the Fatou set can also be defined using equicontinuity (\cite{FS2}).

\deff 
 Let $H: {\mathbb P}^2\to {\mathbb P}^2$ be a rational map such that ${\mathbb P}^2\setminus E$ is nonempty. A point $p \in \mathbb{P}^2\setminus E$ is said to be a Fatou point if there exists for every $\epsilon > 0$ some neighborhood $U$ of $p$ such that $diam H^{\circ n} (U \setminus E) < \epsilon$ for all $n$, where ${diam} S$ stands for the diameter of a subset $S$ in ${\mathbb P}^2$ with respect to the Fubini-Study metric. The Fatou set $\mathcal F (H)$ is the set of Fatou points of $H$. The Julia set ${\mathcal J}(H)$ is the complement $\mathbb{P}^2\setminus \mathcal F (H)$.

\begin{Rem} We note here that it is shown in \cite{Ue} that a point $p$ is in ${\mathcal F}(H)$ if and only if there exists a neighborhood $V$ of $p$ such that the sequence $\{H^{\circ n}\}$ converges normally on $V$.
\end{Rem}
It is not hard to see that the Fatou set is open and the Julia set is closed. The extended indeterminacy set $E$ is clearly  inside the Julia set. Lastly, the map $H$ preserves the Julia as well as the Fatou set. It is worth pointing out that there exists a rational map $H: {\mathbb P}^2\to {\mathbb P}^2$ whose Julia set is equal to ${\mathbb P}^2$, for example the following map from \cite{Forn}: \[H([z_0: z_1: z_2])=[(z_0-z_2)^2: (z_0-2z_2)^2: z_0^2].\] 
Julia set has been extensively studied for the H\'{e}non maps defined by 
\[H([z_0:z_1:z_2])=[z_0^2+xz_2^2+yz_1z_2:z_0z_2:z_2^2],\]
where $x$ and $y$ are real parameters.
We refer the readers to \cite{Forn,For,FS2,He, LM} for more relevant information about complex dynamics in several variables. The main concern here is the map
\begin{equation}
\label{eq: FP}
F([z_0: z_1: z_2]) = [ z_0(z_0^2-z_1^2-z_2^2): z_1^2 z_2:  z_2(z_0^2-z_2^2)]
\end{equation}  
defined through the self-similarity of the Koopman representation of $D_\infty$ (cf. (\ref{eq: F})).

\section{Dynamical properties of $F$}

This section shall establish the main theorem of this paper. For a general rational map $H:\mathbb{P}^2\to \mathbb{P}^2$, it is often important but rather technical to determine whether the extended inderminacy set $E$ is a proper subset of $\mathbb{P}^2$. Since $E$ is a subset of the Julia set ${\mathcal J}(H)$, if $E=\mathbb{P}^2$ then naturally ${\mathcal J}(H)=\mathbb{P}^2$ which is not an interesting case. The first subsection will look into this issue for the map $F$.

\subsection{Extended indeterminacy set}

Consider the map $F: \mathbb{P}^2\to \mathbb{P}^2$ defined in (\ref{eq: FP}), and let $I_n$ be the indeterminacy set of $F^{\circ n}$ in Definition 4.8. The first two indeterminacy sets $I_1$ and $I_2$ are not hard to determine, but since they are crucial to the subsequent discussions we include the details here.
\begin{lem} 
 \label{prop: orbits}
 $ I_1= \{ [\pm1:1:0], [0:1:0], [\pm 1: 0:1] \}. $
\end{lem}
\begin{proof}
By definition, the first indeterminacy set
\begin{align*}
I_1 &= \{z \in \mathbb{P}^2 \mid F(z) = (0,0,0) \} \\
	&= \{z\in \mathbb{P}^2 \mid  z_0(z_0^2-z_1^2-z_2^2)=0, z_1^2z_2=0, z_2(z_0^2-z_2^2)=0\}.
\end{align*}
Then either $z_1=0$ or $z_2=0$ from the middle condition. If $z_1=0$ then
\[z_0(z_0^2-z_2^2)=0 \quad \text{ and } \quad z_2 (z_0^2-z_2^2)=0.\] Hence either $z_0^2-z_2^2=0$ or $z_1=z_2=0.$
This shows that $[\pm 1:0:1]\in I_1$. If $z_2=0$ then by the first condition we have $z_0(z_0^2-z_1^2)=0,$ which occurs if either $z_0=0$ or $z_0^2=z_1^2$. This gives us the remaining points of $I_1$.
\end{proof}
Since $I_1$ is a finite set, the map $F$ is a rational map.

\begin{lem}  
\label{prop: I2}
$I_2 = I_1 \cup \{  [\pm 1: \zeta :1] \mid \zeta \in \mathbb{C} \}.$
\end{lem}
\begin{proof}
Since $I_2  = \{z \in \mathbb{P}^2 \mid F(z) \in I_1\}$, we must consider pre-images to each point in $I_1$ with respect to $F$. First we consider pre-images to the points $[\pm 1:0:1]$. If $F(z) = [\pm 1:0:1]$ then
\[   z_0(z_0^2-z_1^2-z_2^2)=\pm 1, \quad  z_1^2z_2=0,  \quad z_2(z_0^2-z_2^2)=1.\]
It is clear from the third condition that $z_2\neq 0$. Hence $z_1=0$ from the middle condition. Then the remaining conditions are
\[  z_0(z_0^2-z_2^2)=\pm 1 \quad \text{ and }   \quad z_2(z_0^2-z_2^2)=1,\]
which means $z_0 = \pm z_2 \neq 0.$
But this would contradict with the third equation. Thus there is no pre-image to the points $[\pm1:0:1]$.

Next, suppose $z$ were a pre-image of $[0:1:0]$. Then
\[ z_0(z_0^2-z_1^2-z_2^2)=0, \quad z_1^2z_2=1, \quad z_2 (z_0^2-z_2^2)=0.\] 
If $z_0=0$ then the third equation would imply $z_2=0$, which would cause a contradiction with the middle equation. Therefore we must have $ z_0^2-z_1^2-z_2^2=0.$
Writing $z_1^2 = z_0^2-z_2^2$ and substituting it into the third equation, we have
\[z_2(z_1^2)=0.\] 
But this contradicts with the condition that $z_1^2z_2=1$. Therefore there does not exist any pre-images to $[0:1:0]$.

Lastly, we must account for any possible pre-image points $z$ to $[\pm1:1:0]$, or equivalently any $z\in \C^3$ such that
\[z_0(z_0^2-z_1^2-z_2^2)=\pm \lambda, \quad  z_1^2z_2=\lambda,  \quad z_2(z_0^2-z_2^2)=0,\]
for some $\lambda\neq 0$. The second and third equations above imply $z_0^2-z_2^2=0$, i.e., $z_0=\pm z_2$. Applying this to the first equation yields $z_0z_1^2=\pm\lambda$ which is equivalent to the second equation. Projecting such $z$ to ${\mathbb P}^2$, one sees that the set \[\{ z\in \mathbb{P}^2 \mid z= [\pm1: \zeta :1] ,  \zeta \neq 0 \}\] is the set of pre-images to the point $[\pm 1:1:0]$.
\end{proof}

To determine $I_n$ for $n \geq 3$, we shall need the function $\hat{\tau}$ defined in (\ref{eq: tauhat}). One sees that since $\hat{\tau}$ is homogeneous of degree $0$, it can be considered as a map from ${\mathbb P}^2$ to $\hat{\C}$ which we shall denote by $\tau$. Observe that now $\tau$ is holomorphic on the subset ${\mathbb P}^2\setminus \phi(S_0)$. Then the following proposition follows directly from (\ref{eq: Tche}).

\begin{prop}
\label{prop: pre-diagram}
For $k \geq 1$ the following diagram is commutative: 
\begin{displaymath}
  \xymatrix{
    {\mathbb{P}^2 \setminus \displaystyle {I}_{k}}\ar[r]^{F^{\circ k}}  \ar[d]^{\tau}
    &  \mathbb{P}^2  \ar[d]^{\tau} \\
    {\hat{\mathbb{C}}} \ar[r]^{T^{\circ k}} 
    & {\hat{\mathbb{C}}}.
  }
\end{displaymath}
\end{prop}

An important method to simplify the computation of $F^{\circ n}$ is to use the function $\T(z)$ to write
\begin{align}
 F(z) &= [2\T(z) z_0z_1z_2  : z_1^2z_2: z_1z_2( 2\T(z)z_2 +z_1)]\\
&=[2\T(z)z_0:z_1:2\T(z) z_2 +z_1],\ \ \ z\in {\mathbb P}^2\setminus I_1.
\label{eq: F w T}
\end{align}

\Rem Some justification needs to be made about (\ref{eq: F w T}) for the case $z_1z_2=0$. If $z_1=0$ then $z_0^2\neq z_2^2$ since $z\notin I_1$. Hence by (\ref{eq: FP}) we have 
\[F(z)=[z_0(z_0^2-z_2^2): 0: z_2(z_0^2-z_2^2)]=[z_0: 0: z_2].\]
In this case since $\T(z)=\infty\neq 0$, if one computes $F(z)$ using (\ref{eq: F w T}) and factors out $\T(z)$ then one also gets $[z_0: 0: z_2]$. If $z_2=0$, then $z_0^2\neq z_1^2$ since $z\notin I_1$. Hence by (\ref{eq: FP}) we have $F(z)=[z_0(z_0^2-z_1^2): 0: 0]=[z_0: 0: 0]$.
If we evaluate $F(z)$ using (\ref{eq: F w T}) and factor out $\T(z)=\infty$, then we also get 
$[z_0: 0: 0]$. In short, (\ref{eq: F w T}) holds for all $z\notin I_1$.\\

Using $I_1$ and $I_2$ we then describe the indeterminacy set for any number of iterations. 
\begin{prop} The indeterminacy set $I_n$ when $n\geq 3$ is
\label{prop: In'}
\[  \medmath{I_n = I_{n-1} \cup \Bigg\{ z \in \mathbb{P}^2 \mid z= \Bigg[ \frac{\pm1}{ \eta} : \zeta : \frac{1}{\eta } - \zeta \sum_{j=1}^{n-2} \frac{1}{2^{j} \displaystyle  \prod_{k=0}^{j-1} T^{\circ k}(\T (z))} \Bigg], \zeta \in\mathbb{C} \Bigg\}}, \]
where $\eta = 2^{n-2}  \displaystyle \prod_{k=0}^{n-3} T^{\circ k}(\T(z))$ and $z \in I_n \setminus I_{n-1}$.
\end{prop}
Although this is not a concrete way to describe $I_n$, it will be sufficient for us to obtain some larger results later.
\begin{proof}
Before proceeding we introduce the notation $I_n' = I_n \setminus I_{n-1}$ to exclusively focus on new indeterminacy points. For example $I_2' = \{ [\pm 1: \zeta : 1] \mid \zeta \neq 0\}$. First we examine $I_3'= \{z \in \mathbb{P}^2 \mid F(z) \in I_2'\}$, and use (\ref{eq: F w T}) to obtain equations
\[ 2 \T(z) z_0 = \pm 1 , \qquad z_1 = \zeta , \qquad 2 \T(z) z_2 +z_1 = 1.\]
Clearly $\T(z) \neq 0$ in this case, hence
\[I_3' = \bigg\{ z\in \mathbb{P}^2 \mid z=\bigg[ \frac{\pm1}{2\T(z)} : \zeta: \frac{1-\zeta}{2 \T(z)}\bigg], \zeta \in \mathbb{C}, \zeta\neq 0 \bigg\}. \]
We shall repeat this process to discover the pre-images to the points of $I_n'$. But before proceeding note that when observing a pre-image there is a shift in notation, specifically if $w\in I_3'$ and $F(z)=w$ then $\T(w) = \T(F(z))= T(\T(z))$. Let
\[I_4' = \{z\in \mathbb{P}^2 \mid  F(z) \in I_3' \}.\]
 Then for $z \in I_4'$ the following system of equations must be satisfied:
\[2 \T(z) z_0 = \frac{\pm1}{2T(\T(z))} , \qquad z_1 = \zeta , \qquad 2z_2 \T(z) + z_1 = \frac{1-\zeta}{2 T(\T(z))}.\]
Solving for $z_0, z_1$ and $z_2$, we have
\[ I_4' = \{ z \in \mathbb{P}^2 \mid z=\Big[ \frac{\pm1}{2^2 \T(z) T(\T(z))} : \zeta : \frac{1-\zeta}{2^2 \T(z) T(\T(z))} - \frac{\zeta}{2 \T(z)}\Big] , \zeta \in \C, \zeta\neq 0\}. \]
Then one may use induction to demonstrate that $I_n'$ is of the form
\[\medmath{  \Bigg\{ z= \Big[ \frac{\pm1}{ \beta(z)} : \zeta : \frac{1}{\beta(z) } - \zeta \sum_{j=1}^{n-2} \frac{1}{2^{j}   \prod_{k=0}^{j-1} T^{\circ k}(\T (z))} \Big] \mid \zeta \in\mathbb{C}, \zeta\neq 0 \Bigg\}}\]
for all $n\geq 3$, where $\beta(z) = 2^{n-2}   \prod_{k=0}^{n-3} T^{\circ k}(\T(z))$.
Assume that $w\in I_n'$ is of the above form and $z \in I_{n+1}'$ such that $F(z) = w$. Then
\begin{align*}
F(z) &= [2 \T(z) z_0: z_1 : 2 \T(z) z_2 +z_1] = [w_0: w_1:w_2] \\
&= \medmath{\bigg[ \frac{\pm1}{\beta(w)} : \zeta : \frac{1}{\beta(w)} - \zeta \sum_{j=1}^{n-2} \frac{1}{2^{j} \prod_{k=0}^{j-1} T^{\circ k}(\T(w))} \bigg] }.
\end{align*}
Since $F(z) = w$, we can re-write $F(z)$ above as:
\[\medmath{\bigg[ \frac{\pm1}{\beta (F(z))} : \zeta : \frac{1}{\beta (F(z))} - \zeta \sum_{j=1}^{n-2} \frac{1}{2^{j} \prod_{k=0}^{j-1} T^{\circ k}(\T(F(z)))} \bigg] }.\]
Now comparing each component individually one has $z_1 = \zeta$,
\[ 2 \T(z) z_0 = \frac{\pm1}{2^{n-2} \prod_{k=0}^{n-3} T^{\circ k}(\T(F(z)))}, \]
which can be re-written using the fact $\tau(F(z))=T(\tau (z))$ as \[ z_0 = \frac{\pm1}{2^{n-1} \prod_{k=0}^{n-2} T^{\circ k}(\T(z))},\]
and
\[ 2 \T(z) z_2 + \zeta = \frac{1}{2^{n-2} \prod_{k=0}^{n-3} T^{\circ k}(\T(F(z))} - \zeta \sum_{j=1}^{n-2} \frac{1}{2^{j} \prod_{k=0}^{j-1} T^{\circ k}(\T(F(z)))}. \]
Solving $z_2$ from the last term, we obtain
\[z_2 =  \frac{1}{2^{n-1} \prod_{k=0}^{n-2} T^{\circ k}(\T(F(z))} - \zeta \sum_{j=1}^{n-1} \frac{1}{2^{j} \prod_{k=0}^{j-1} T^{\circ k}(\T(z))}.\]
Thereby if $z \in I_{n+1}'$ then $z$ is a point of the form:
\[\medmath{\bigg[\frac{\pm1}{2^{n-1} \prod_{k=0}^{n-2} T^{\circ k}(\T(z))}: \zeta : \frac{1}{2^{n-1} \prod_{k=0}^{n-2} T^{\circ k}(\T(z))} - \zeta \sum_{j=1}^{n-1} \frac{1}{2^{j} \prod_{k=0}^{j-1} T^{\circ k}(\T(z))} \bigg]}\]
for some nonzero $\zeta \in \C$, which completes the proof.
\end{proof}

Now we are in position to prove the following important fact. Recall that the extended indeterminacy set $E$ is the closure of $\cup_{k=1}^\infty {I}_k.$
\begin{lem} $E \neq \mathbb{P}^2.$
\end{lem}
\begin{proof} It is not initially apparent that $E$ is not the entirety of $\mathbb{P}^2$ so we demonstrate it in detail. To do so we show there exists an open set in $p^c(A_\pi)$ that is not contained in $ E$. A key ingredient in the proof is the following function
\begin{equation}
\label{eq: fn}
f_n(z) = \sum_{j=1}^{n-1} \frac{1}{2^{j} \prod_{k=0}^{j-1} T^{\circ k}(\T(z))}, \ n \geq 2, z\in p^c(A_\pi).
\end{equation} 
The function $f_n(z)$ is well defined on $p^c(A_\pi)$ for two reasons: 
\begin{enumerate}
\item In this case $ \T(z) \notin [-1,1]$ by Theorem \ref{thm: PA}, 
\item $T^{\circ k}(\T(z))\notin [-1, 1]$ for all $k$ because of Theorem \ref{thm: Tch}.
\end{enumerate} 
Moreover, since $T(x)$ is holomorphic on $\hat{\C}$ and $\tau$ is holomorphic on the set ${\mathbb P}^2\setminus \phi(S_0)$ which contains $p^c(A_\pi)$, we see that $f_n$ is holomorphic on $p^c(A_\pi)$.
Since $\T(z)$ is holomorphic on $p^c(A_\pi)$, for every compact subset $K\subset p^c(A_\pi)$ the image $\T(K)$ is compact in $\hat{\C}\setminus [-1, 1]$. Thus $\{T^{\circ k}(\T(z))\}$ converges uniformly to $\infty$ on $K$ by Theorem \ref{thm: Tch}. This implies that the sequence $\{f_n\}$ converges normally on $ p^c(A_\pi)$. Define for $z \in p^c(A_\pi)$ that 
\begin{equation}
\label{eq: fz}
f(z) = \lim_{n \to \infty} f_n(z),
\end{equation}
then $f$ is holomorphic on $p^c(A_\pi)$. 

Now we turn to the indeterminacy sets. Observe that Proposition \ref{prop: In'} implies that for all $n \geq 3$ we have
\begin{equation}
\label{eq: In' subset}
I_n' \subset \big\{z \in p^c(A_\lb) \mid z_2 = \pm z_0 - z_1 f_{n-1}(z) \big\},
\end{equation}
which is an analytic set in $p^c(A_\lb)$. Since
\begin{equation}
\label{eq: In from In'}
 I_n = I_n' \cup I_{n-1}' \cup I_{n-2}' \cup \cdots \cup I_2,
\end{equation}
by (\ref{eq: In' subset}) we have
\[ I_n \subset \Big( I_2 \cup \bigcup_{k=2}^{n-1} \{z_2 = \pm z_0 -z_1 f_k(z) \} \Big), \text{ when} \ n\geq 3.\] 
Now we are ready to construct an open subset in ${\mathbb P}^2$ that is not contained in the extended indeterminacy set $E$. First, for all $z \in p^c(A_\pi)$ such that $|\T(z)| >1$ we have
\[|T(\T(z))| = |2 (\T(z))^2-1| \geq 2|\T(z)|^2-1 >1,\]
and subsequently $|T^{\circ n}(\T (z))|>1$ for all $n\geq 1$. Applying this fact to (\ref{eq: fn}) allows us to conclude that if $n \geq 2$ then
\[ |f_n(z)| < \frac{1}{2} + \frac{1}{4} + \cdots +\frac{1}{2^{n-1}} <1.\]
Therefore, by (\ref{eq: In' subset}) if such $z$ is in $I_n$ for some $n\geq 3$ then $z$ must satisfy
\begin{equation}
\label{eq: In cond.}
 |z_2| \leq |z_0|+|z_1|. 
\end{equation}
Observe that Lemma \ref{prop: orbits} and \ref{prop: I2} show that points in $I_1$ and $I_2$ also satisfy the inequality (\ref{eq: In cond.}). This implies that if $z\in p^c(A_\pi)$ is such that 
$|\T(z)|>1$ and $z\in E$, then $ |z_2| \leq |z_0|+|z_1|$. We therefore have the following inclusion
\begin{equation}
\label{eq: P-E}
V:=\{z\in p^c(A_\lambda)\mid |\T(z)|>1\}\cap \{ |z_2| > |z_0|+|z_1|\}\subset {\mathbb P}^2\setminus E.
\end{equation}
It is not hard to see that the open set $V$ is nonempty. For example, one easily verifies that $[1:1:3] \in V$, and this completes the proof.
\end{proof}

\subsection{The Julia set of $F$}

We are now ready to determine the Julia set of the map $F$ on $\mathbb{P}^2$.
Recall that in $\mathbb{P}^2$ the projective spectrum of the dihedral group with respect to the Koopman representation is
\[p(A_\pi) =\bigcup_{-1 \leq x \leq 1} \{z\in \mathbb{P}^2 \mid  z_0^2-z_1^2-z_2^2-2z_1z_2 x=0\}. \]
Recall also that $I_\infty =  \bigcup_{j=1}^{\infty} I_j\subset E$. Then
the next proposition follows readily from Proposition \ref{prop: pre-diagram}.
\begin{prop}
\label{prop: diagram}
The following diagram is commutative 
\begin{displaymath}
  \xymatrix{
    {\mathbb{P}^2 \setminus I_\infty}\ar[r]^{F}  \ar[d]^{\tau}
    &  {\mathbb{P}^2 \setminus I_\infty}  \ar[d]^{\tau} \\
    {\hat{\mathbb{C}}} \ar[r]^{T} 
    & {\hat{\mathbb{C}}}.
  }
\end{displaymath}
\end{prop}
Proposition \ref{prop: diagram} and the expression (\ref{eq: F w T}) enable us to rewrite $F^{\circ  n}$ as follows.
\begin{lem}
\label{lem: Fn}
For $n \geq 2$ and $z\in {\mathbb P}^2\setminus  I_\infty$ one has
\[ F^{\circ n} (z)=\medmath{\bigg[2^n z_0 \prod_{k=0}^{n-1} T^{\circ k} ( \T(z) ): z_1: 2^n z_2 \prod_{k=0}^{n-1} T^{\circ k} ( \T(z) )   + z_1 \Big(1 + \sum_{k=1} ^{n-1} 2^k \prod_{i=1}^{k} T^{\circ (n-i)} ( \T(z) )\Big)  \bigg] }.\]
\end{lem}
\begin{proof}
If $z \in \mathbb{P}^2 \setminus I_\infty$ then by  (\ref{eq: F w T}) 
\begin{align*}
F(z)=[2 \T (z) z_0: z_1: 2 \T (z)z_2 + z_1]:= [z_0': z_1': z_2'],
\end{align*}
and hence using the fact that $\T (z') = \T(F(z)) = T (\T(z))$ we obtain
\begin{align}
\label{eq: Fof2a}
F^{\circ 2} (z) &= [2 \T (z') z_0': z_1': 2 \T (z')z_2' + z_1']\\
&=[2^2 \T(z) T( \T (z)) z_0: z_1: 2^2\T(z) T( \T (z))z_2 + z_1(1+2 T(\T(z)) )].
\end{align}
To prove the lemma by induction on $n$, we assume
\begin{align*}
 F^{\circ n} (z) &= [2 T^{\circ (n-1)} ( \T (z)) z_0^{(n-1)}: z_1: 2 T^{\circ (n-1)} ( \T (z)) z_2^{(n-1)}+ z_1]  \\
 &= [z_0^{(n)}:z_1^{(n)}:z_2^{(n)}].
\end{align*}
Then 
\begin{align*}
 F^{\circ (n+1)} (z) &= [2 T( \T (z^{(n)})) z_0^{(n)}: z_1^{(n)}: 2 T( \T (z^{(n)}))z_2^{(n)} + z_1^{(n)}] \\
&= [2 T^{\circ n} ( \T ) z_0^{(n)}: z_1: 2 T^{\circ n} ( \T (z)) z_2^{(n)}+ z_1].
\end{align*}
Substituting $z_0^{(n)}$ and $z_2^{(n)}$ using the induction assumption, we can re-write $F^{\circ (n+1)} (z)$ as
\[\medmath{ [2^2 T^{\circ n} ( \T (z))T^{\circ (n-1)} ( \T(z) ) z_0^{(n-1)}: z_1: 2 T^{\circ n}(\T(z) ) \big( 2T^{\circ (n-1)} ( \T(z) ) z_2^{(n-1)}+z_1\big)+ z_1] .}\]
Applying the process iteratively one establishes the lemma.
\end{proof}

Now we are equipped to reveal a connection between the Fatou set of $F$ and the projective resolvent set $p^c(A_\pi)$.
\begin{lem}  $p^c(A_\pi) \setminus E \subset {\mathcal F}(F) $. 
\label{thm: fatou}
\end{lem}
\begin{proof}
As observed before, in view of Theorem \ref{thm: PA} if $z \in p^c(A_\pi)$ then $\T(z) \notin [-1,1]$, and Theorem \ref{thm: Tch} implies that $T^{\circ n} (\T(z)) \notin [-1,1]$, in particular $T^{\circ n}(\T(z)) \neq 0$ for all $n\geq 1$.
Using Proposition \ref{lem: Fn} and factoring out $2^n \prod_{k=0}^{n-1} T^{\circ k} ( \T(z) )$ from each component of $F^{\circ n}$, one can write
\begin{align}
F^{\circ n} (z) &= \left[ z_0:\dfrac{z_1}{2^n  \prod_{k=0}^{n-1} T^{\circ k} ( \T (z) )} : z_2   + z_1\sum_{k=1}^{n-1} \dfrac{2^{k-n}}{ \prod_{j=0}^{k} T^{\circ j} ( \T (z))}\right] \\
 &=\left [z_0:\dfrac{z_1}{2^n  \prod_{k=0}^{n-1} T^{\circ k} ( \T (z) )} : z_2 +z_1 f_n(z) \right],
\label{eq: factored Fn}
\end{align}
where $f_n$ is as defined in (\ref{eq: fn}). Since $\T: p^c (A_\pi)\to \hat{\C}\setminus[-1,1]$, Theorem \ref{thm: Tch} implies that $T^{\circ n}(\T(z)) \to \infty$ normally on $p^c (A_\pi)$, and hence the middle component in (\ref{eq: factored Fn}) tends to $0$ normally as $n\to \infty$. Moreover, since $f_n$ converges normally to $f$ on $p^c(A_\pi)$ by (\ref{eq: fz}), the iterations
$F^{\circ n}$ converges normally on $p^c(A_\pi)\setminus E$ and 
\[ \lim_{n \to \infty} F^{\circ n}(z) = [z_0:0:z_2 +z_1 f(z)]:= F_\ast ([z_0:z_1:z_2]),\ \ z\in p^c(A_\pi)\setminus E. \]
This completes our proof relating $p^c(A_\pi)$ to $\mathcal{F}(F)$.
\end{proof}

Observe that Lemma \ref{thm: fatou} implies that the Julia set ${\mathcal J}(F)$ is contained in $p(A_\pi)\cup E$. Soon we will conclude that this inclusion is in fact an equality. We shall need a bit more preparation before we proceed to examine the Julia set. First, one observes that if $z \in p(A_\pi)$ then $x: = \T(z) \in [-1,1]$, and hence we can rewrite $T(x)$ in terms of cosine by setting $x=\cos \theta$, where $\theta=\cos^{-1}(x)\in [0, \pi]$. Then $T (x) = \cos( 2\theta)$ and hence $T ^{\circ n} (x) = \cos (2^n \theta)$. The $k$-th Tchebyshev polynomial of the first kind $T_k(x)$ is defined by the equality $\cos (k\theta)=T_k(\cos \theta).$ In other words, the polynomial $T_k$ is semi-conjugate to the multiplication by constant $k$. Hence
\begin{equation} 
T^{\circ n}(\cos (\theta))= T_{2^n} (\cos (\theta)).
\end{equation}
The following two well-known properties of the Tchebyshev polynomials (\cite{MH}) are relevant to our study. 
\begin{prop} \label{prop: Tch}
Let $T_k$ stand for the $k$-th Tchebyshev polynomial on ${\mathbb R}$.
\begin{enumerate}[(a)]
\item $2T_m(x) T_n(x) = T_{m+n}(x) + T_{|m-n|}(x). $

\item $\prod_{j=0}^n T^{\circ j}(x) = \frac{1}{2^n} \sum_{k=1}^{2^n} T_{2k-1}(x) .$
In particular, 
if $\frac{\cos^{-1}(x)}{\pi} \notin \mathbb{Z}$, then
 \begin{equation}\label{eq: sine}
\prod_{j=0}^{n} T^{\circ j}(x) = \frac{ \sin (2^{n+1} \cos^{-1}(x))}{2^{n+1}\sin \left(\cos^{-1}(x)\right)}.
\end{equation}
 \end{enumerate}
\end{prop}
These properties shall enable us to express the iterations of $F^{\circ n}$ on the projective spectrum $p(A_\pi)$ using trigonometric functions. We are now ready to state and prove the main theorem of this paper.
\begin{thm}\label{thm: main}
Consider the infinite dihedral group $D_\infty$ and the associated map $F$ defined in (\ref{eq: FP}) through the self-similarity of its Koopman representation $\pi$, and let $E$ be the extended indeterminacy set of $F$. Then ${\mathcal J}(F)=p(A_\pi) \cup E$ .
\end{thm}

\begin{proof}
First of all, Lemma \ref{thm: fatou} implies that ${\mathcal J}(F)\subset p(A_\pi)\cup E$. So it only remains to show the inclusion in the opposite direction. The inclusion $E\subset {\mathcal J}(F)$ is a general fact. Recall from Proposition \ref{lem: Fn} that when $z\in {\mathbb P}^2\setminus E$ we have
\[F^{\circ n}(z)=\bigg[2^n z_0 \prod_{k=1}^n T^{\circ k} ( \T ): z_1: 2^n z_2 \prod_{k=1}^n T^{\circ k} ( \T ) + z_1 \big(1 + \sum_{k= 1} ^{n} 2^k \prod_{i=0}^{k-1} T^{\circ (n-i)} ( \T )\big)  \bigg] .\]
If $z\in p(A_\pi)$ then Theorem \ref{thm: PA} implies that $\tau(z)=\cos \theta$ for some $0\leq \theta \leq\pi$. Suppose $\frac{\theta}{\pi}$ is non-dyadic, i.e., $2^n\frac{\theta}{\pi}\notin {\mathbb Z}$ for any integer $n\geq 0$. Then by Proposition \ref{prop: Tch}, we have
\[ F^{\circ n}(z) = \bigg[z_0  \frac{\sin (2^n \theta)}{\sin \theta} : z_1: z_2 \frac{\sin (2^n \theta)}{\sin \theta} + z_1(1+ \sum_{k=1}^{n-1}  \frac{\sin(2^n \theta)}{2^{n-1}\sin (2^{k+1} \theta)})\bigg], \]
which does not converge as $n\to \infty$. We now show that there exists no open neighborhood of any point $p\in p(A_\pi)\setminus E$ on which the sequence $\{ F^{\circ n}\}$ is normal. To this end we suppose $p=[p_0: p_1: p_2]$ is an arbitrary point in $p(A_\pi) \setminus E$. Then in particular $p\notin I_1$ and consequently by Theorem \ref{thm: PA} we must have $p_1p_2\neq 0$. For any open neighborhood $V$ of this $p$ in ${\mathbb P}^2\setminus E$ there exists a small open neighborhood
$V_0$ of $p$ such that $z_1z_2\neq 0$ for all $z\in V_0\subset V$. Therefore $\tau(z)$ is holomorphic on $V_0$ and thus $\tau(V_0)$ contains an open neighborhood $V_1$ of $\tau(p)$ in $\hat{\C}$. Since $\tau(p)\in [-1, 1]$, the intersection $V_1\cap [-1, 1]$ contains an interval, say $(a, b)$. By the density of non-dyadic numbers in ${\mathbb R}$, there exists a number $\xi\in (a, b)$ such that $\frac{\cos^{-1}(\xi)}{\pi}$ is non-dyadic. Since $(a, b)\subset V_1\cap [-1, 1]  \subset \tau(V_0)$, there exists $q\in V_0$ such that $\tau(q)=\xi$. Then the sequence $\{ F^{\circ n}(q)\}$ is not convergent by the foregoing argument, and therefore the sequence $\{F^{\circ n}\}$ is not normal on $V$. This concludes that $p(A_\pi)\setminus E$ is a subset of the Julia set ${\mathcal J}(F)$. Since $E\subset {\mathcal J}(F)$, we have $p(A_\pi)\cup E\subset {\mathcal J}(F)$ and this completes the proof.
\end{proof}

The following corollary follows immediately from Theorem \ref{thm: main} and the proof of Lemma \ref{thm: fatou}.

\begin{cor}\label{thm: ff}
For the map $F$ defined in (\ref{eq: FP}) we have
\begin{enumerate}[(a)]
\item the Fatou set ${\mathcal F}(F)=p^c(A_\pi)\setminus E$;
\item the iteration sequence $\{ F^{\circ n}\}$ converges normally on ${\mathcal F}(F)$ to the function \[F_{*}([z_0: z_1: z_2]=[z_0: 0: z_2+z_1f(z)],\] where $f$ is as defined in (\ref{eq: fz}).
\end{enumerate}
\end{cor}

\subsection{The limit function $f$} Much to our surprise, the limit function $f$ defined in (\ref{eq: fz}) can be determined explicitly. We shall address this fact in this subsection. First, the following lemma is easy to prove by induction.

\begin{lem} \label{lem: lf}
For $\xi\in {\mathbb C}$ and any natural number $n\geq 2$, we have
 \[\sum_{k=1}^{n-1} \csc(2^{k}\xi)= \cot(\xi)-\cot(2^{n-1} \xi).\]
\end{lem}
For simplicity, we set $S=\phi(S_0)=\{z\in {\mathbb P}^2: z_0^2-z_1^2-z_2^2=0\}$.
\begin{thm}\label{thm: ef}
Let $f$ be the limit function defined in (\ref{eq: fz}). Then
\[f(z)=\tau(z)-i\sqrt{1-\tau^2(z)},\ \ z\in p^c(A_\pi)\setminus E.\]
In particular, it can be extended holomorphically to ${\mathbb P}^2\setminus S$.
\end{thm}

\begin{proof} 
Recall that since
\[\cos \xi=\frac{1}{2}(e^{i\xi}+e^{-i\xi}),\ \ \ \ \sin \xi=\frac{1}{2i}(e^{i\xi}-e^{-i\xi})\]
are entire functions on $\C$, the function $\cos \xi$ has a local inverse function $\cos^{-1}$ when $\xi$ is not an integer multiple of $\pi$. Moreover, since trigonometric identities of $\sin x$ and $\cos x$ hold for complex numbers $x$, it is not hard to check that equality (\ref{eq: sine}) holds for all complex numbers $\xi$ for which $\cos^{-1} \xi$ is well-defined. Suppose $z\in p^c(A_\pi)\setminus E$ is such that $\frac{\theta(z)}{\pi}:=\frac{\cos^{-1}(\tau(z))}{\pi}$ is non-dyadic. Then the functions $f_n$ defined in (\ref{eq: fn}) can be written as 
\[f_n(z)=\sum_{k=1}^{n-1}\frac{\sin\left(\theta(z)\right)}{\sin\left(2^{k}\theta(z)\right)},\ \ n\geq 2.\]
And hence by Lemma \ref{lem: lf} we have
\begin{equation} \label{eq: fi}
\frac{f_{n}(z)}{\sin(\theta(z))}=\sum_{k=1}^{n-1} \csc(2^{k}\theta(z))= \cot(\theta(z))-\cot(2^{n-1} \theta(z)).
\end{equation}
It is established in Section 5.1 that the sequence $\{f_n\}$ converges normally as $n\to \infty$ on $p^c(A_\pi)\setminus E$. Hence the sequence $\{\cot(2^n\theta(z))\}$ converges normally. To determine the limit of this sequence, we observe through Equality (\ref{eq: sine}) that 
\[ \frac{ \sin\theta(z)}{\sin \left(2^{n+1}\theta(z)\right)}=\frac{1}{2^{n+1}\prod_{j=0}^{n} T^{\circ j}(\tau(z))},\ \ n\geq 2.\]
Since $z\in p^c(A_\pi)$, we have $\tau(z)\in \hat{\C}\setminus [-1,1]$ and consequently $T^{\circ j}(\tau(z))\to \infty$ as $j\to \infty$ in view of Theorem \ref{thm: Tch}. It follows that $\sin \left(2^{n+1}\theta(z)\right)\to \infty$ as $n\to \infty$. Then 
\begin{align*}
s:=\lim_{n\to \infty}\cot (2^n\theta(z))&=\lim_{n\to \infty}\frac{\pm \sqrt{1-\sin^2(2^n\theta(z))}}{\sin(2^n\theta(z))}\\
&=\pm i.
\end{align*}
Of course, we need to determine whether $s=i$ or $s=-i$. Letting $n\to \infty$ in Equation (\ref{eq: fi}),
we have 
\begin{align*}f(z)&=\cos(\theta(z))-s\sin(\theta(z))\\
&=\tau(z)-s\sqrt{1-\tau^2(z)}=\tau(z)\pm \sqrt{\tau^2(z)-1}.
\end{align*}
On the other hand, since 
\begin{align*}
f(z) &= \lim_{n \to \infty} f_n(z) \\
&= \frac{1}{2 \T(z)} \bigg( 1 + \frac{1}{2 T(\T(z))} + \frac{1}{2^2 T(\T(z)) T^2(\T(z))} + \cdots \bigg),
\end{align*}
and $\tau(z)=\infty$ at points $z\in p^c(A_\pi)$ such that $z_1z_2=0$, we must have $f(z)=0$ at points $z$ where $\tau(z)=\infty$. This concludes that $s=i$ and hence 
\[f(z)=\tau(z)- i\sqrt{1-\tau^2(z)}.\]

Furthermore, since in view of its definition $\tau$ is holomorphic on ${\mathbb P}^2\setminus S$, the function $f$ extends holomorphically to ${\mathbb P}^2\setminus S$. Further, for points $z\in S$ one has $\tau(z)=0$ and therefore $f(z)=-i$. This indicates that $f$ can be extended to the entire ${\mathbb P}^2$.
\end{proof}

\Rem Some observations about the function $f$ are worth mentioning.

(1) Since $\tau(z)=\infty$ for $z\in S^c\cap \{z_1z_2=0\}$, the function $f$ vanishes there and it has no other zeros by Theorem \ref{thm: ef}.

(2) For every $z\in S$ we have $\tau(z)=\infty$ by definition, and hence $f(z)=-i$. This means that $f$ is well-defined on the entire ${\mathbb P}^2$. However, as indicated in Section 3.2, if one consider the sequence $\xi_n=[1+\frac{1}{n}:0:1], n\geq 1$ then
$\tau(\xi_n)=\infty$ and hence $f(\xi_n)=0$. But $f([1:0:1])=-i$. This indicates that $f$ can have discontinuities at points in $S\cap \{z_1z_2=0\}$. 

(3) In the case $z\in S_{\cos \theta}\subset p(A_\pi)$, Theorem \ref{thm: PA} implies that $\tau(z)=\cos \theta$ for some $\theta \in [0, 2\pi)$ and hence $f(z)=cos\theta -i\sin \theta=e^{-\theta i}$. In particular, this shows that $f(p(A_\pi))={\mathbb T}$. Moreover, this is consistent with the earlier fact that when $z\in S$ we have $f(z)=-i=e^{-\pi i/2}$. This observation leads to the following corollary.

\begin{cor}
Let $f$ be defined as in Theorem \ref{thm: ef}. Then $f^{-1}(\mathbb T) =p(A_\pi)$.
\end{cor}
\begin{proof}
We have observed in Remark 5.15 that $p(A_\pi)\subset f^{-1}(\mathbb T)$. It is left to show the inclusion in the other direction. Suppose $z\in {\mathbb P}^2$ such that $f(z)\in {\mathbb T}$. Since 
\[\tau^2(z)+\left(\sqrt{1-\tau^2(z)}\right)^2=1,\]
there exists a complex number $\theta$ such that $\tau(z)=\cos\theta$. And therefore one can write
$f(z)=e^{-i\theta}$. Write $\theta=x+yi$ for some $x, y\in {\mathbb R}$. Then the fact that $|f(z)|=1$ implies $y=0$, and consequently $\tau(z)=\cos x\in [-1, 1]$, i.e., one has $z\in p(A_\pi)$.
\end{proof}

\section{Some discussions on the group $\mathcal G$}

Now we turn our attention to the Grigorchuk group ${\mathcal G}$ and its Koopman representation $\pi$ on the boundary $\partial T$ of the rooted binary tree. Much study of the group's spectral properties has been done, for instance in \cite{BG, DG, GLS, GNS2}. In particular, the classical spectrum of the pencil $\pi(ta+ub+vc+wd)$, where $t, u, v, w\in {\mathbb R}$, was studied and found to be related to the Cantor set (\cite{GLS}, Theorem 4.2). Since the pencil \[R_\pi(z) = z_0I +z_1\pi(a)+z_2\pi(b)+z_3\pi(c)+z_4\pi(d),\ \ z\in \C^5\]
is more general, it is advantageous to determine the projective spectrum $P(R_\pi)$. This goal still evades us at this point, which prevented us from determining the whole Julia set of the map $G$ defined in (\ref{eq: G(z)}). Nevertheless, there is a connection between $\mathcal G$ and $D_\infty$ which has allowed us a chance for progress. 

\subsection{A partial result about $P(R_\pi)$}
Consider the element $u=\frac{1}{2}(b+c+d-1)$ in the group algebra $\C{[\mathcal G}]$. One checks easily that $u^2=1$. It is shown in \cite{GY} that the projective spectrum of the pencil $\tilde{R}_\pi(z)=z_0I+z_1\pi(a)+z_2\pi(u)$ of the group $\mathcal G$ and that of the pencil $A_\pi(z)=z_0I+z_1\pi(a)+z_2\pi(t)$ of $D_\infty$ are identical, i.e., we have
\begin{equation}
\label{eq: pu}
p(\tilde{R}_\pi)=\bigcup_{-1\leq x \leq 1}\{z\in \mathbb{P}^2 \mid  z_0^2-z_1^2-z_2^2-2z_1z_2x=0\}.
\end{equation}
An interesting application of this result is given in the following example.

\begin{Ex}
The operator $H_\pi=\pi(a+b+c+d)$ is often called the Hecke type operator for $\mathcal G$ with respect to the Koopman representation $\pi$. The classical spectrum $\sigma(H_\pi)$ was determined in \cite{BG} to be $[-2, 0]\cup [2, 4]$. If we write 
\[H_\pi-\zeta I=(1-\zeta)I+\pi(a)+2\pi(u),\]
then the spectrum $\sigma(H_\pi)$ can also be determined by (\ref{eq: pu}). Indeed, substituting $(z_0, z_1, z_2)$ by $(1-\zeta, 1, 2)$ one has
\[(1-\zeta)^2-5-4x=0,\ \ x\in [-1, 1],\]
and the range for $\zeta$ is easily determined to be $[-2, 0]\cup [2, 4]$.
\end{Ex}

Now we consider the subspace $\hat{M} = \{ z \in \mathbb{C}^5 \mid z_2=z_3=z_4 \}$ and define the map ${X}: \mathbb{\C}^3 \to \hat{M} $ by
\begin{equation}\label{eq: X}
(w_0, w_1, w_2) \to \big( w_0 - \frac{w_2}{2}, w_1, \frac{w_2}{2}, \frac{w_2}{2}, \frac{w_2}{2}\big).
\end{equation}
One sees that $X$ is an isomorphism between the two vector spaces.
Using the fact that $u=\frac{1}{2}(b+c+d-1)$ one easily verifies that the pencils $A_\pi$ for $D_\infty$ and $R_\pi$ for ${\mathcal G}$ are related by $A_\pi(w)=R_\pi(X(w)),\ w\in \C^3$. 
\begin{cor}\label{eq: spec}
A point $w$ is in $P(A_\pi)$ if and only if $X(w)$ is in $P(R_\pi)$, i.e., \[X(P(A_\pi))=P(R_\pi)\cap \hat{M}.\]
\end{cor}

Now going back to the fixed points of the map $G$ in the projective space as described in Proposition \ref{prop: Gfix} (c), one can use Corollary \ref{eq: spec} to verify that the set $\phi(Y_1)$ and the point $[-1:-2:1:1:1]$ are inside the spectrum $p(R_\pi)$. The connection between $\phi(Y_2)$ and $p(R_\pi)$ is not clear at this point.

\subsection{A partial result about ${\mathcal J}(G)$}
Interestingly, the map $X$ above also furnishes a semi-conjugacy between the map $F$ in (\ref{eq: F}) associated with the dihedral group and a part of the map $G$ in (\ref{eq: G(z)}) associated with the Grigorchuk group. 

\begin{prop}\label{prop: FG} The following diagram is commutative:
\begin{displaymath}
  \xymatrix{
    {\mathbb{\C}^3 }\ar[r]^{F}  \ar[d]^{{X}}
    &  {\mathbb{\C}^3}  \ar[d]^{{X}} \\
    \hat{M} \ar[r]^{G} 
    & \hat{M}.
  }
\end{displaymath}
\end{prop}
\begin{proof} 
Recall that for $z\in \C^5$ the function $\alpha(z)=(z_0+z_4)^2-(z_2+z_3)^2$. Hence for $(w_0, w_1, w_2)\in \C^3$, we have
\[\alpha ({X}(w))=(w_0 - \frac{w_2}{2}+\frac{w_2}{2})^2 - (w_2)^2 = w_0^2-w_2^2.\]
Hence 
\begin{align*}
G\left({X}(w)\right) =& G\left( w_0 -\frac{w_2}{2}, w_1, \frac{w_2}{2},\frac{w_2}{2},\frac{w_2}{2}\right)\\
=& \left( w_0 (w_0^2-w_1^2-w_2^2) - \frac{w_2}{2} (w_0^2-w_2^2),  w_1^2w_2, \frac{w_2}{2}\alpha, \frac{w_2}{2} \alpha, \frac{w_2}{2} \alpha \right).
\end{align*}
On the other hand, since 
\[F(w) = \big( w_0(w_0^2-w_1^2-w_2^2), w_1^2 w_2, w_2 (w_0^2-w_2^2) \big),\]
one verifies by direct computation that ${X}\left(F(w)\right)=G\left({X}(w)\right)$.
\end{proof}
Proposition \ref{prop: FG} implies that $F(w)={\bf 0}$ if and only if $X(F(w))={\bf 0}$ and thus if and only if $G(X(w))={\bf 0}$. Note that here ${\bf 0}$ refers to the origin of $\C^3$ as well as $\C^5$. Thus when $X$, $F$, and $G$ are considered as maps in projective space and $M=\phi(\hat{M})$, one has the following fact.

\begin{cor} Let $I_n(F)$ and $I_n(G)$ be the $n$-th indeterminacy set for $F$ and $G$, respectively. Then the following diagram commutes for every $n\geq 1:$
\label{eq: GM}
\begin{displaymath}
  \xymatrix{
    {\mathbb{P}^2 \setminus I_n(F)} \ar[r]^{\ \ \ \ \ F}  \ar[d]^{X}
    &  {\mathbb{P}^2}  \ar[d]^{X} \\
   M\setminus I_n(G) \ar[r]^{\ \ \ \ \ G} 
    & M .
  }
\end{displaymath}
\end{cor}
Let $G|_M$ be the restriction of $G$ to the submanifold $M\subset {\mathbb P}^4$. It is not difficult to see that $I_n(G)\cap M=I_n(G|_M)$. However, since the extended indeterminacy set involves taking the closure, it is not clear if $E_G\cap M=E_{G|_M}$. Nevertheless, since the map $X$ is a homeomorphism we have $E_{G|_M}=X(E_F)$. Lemma 5.6 then implies that $X(E_F)\neq M$, and the next corollary follows from Theorem \ref{thm: main} and the relation (\ref{eq: spec}).
\begin{cor}
Let $G|_M$ be the restriction of $G$ to the submanifold $M$. Then its Julia set
${\mathcal J}(G|_M)=X(p(A_\pi))\cup X(E_{F}).$
\end{cor}

We now propose a conjecture to end this round of discussions. Consider the Grigorchuk group ${\mathcal G}$. Recall that $p(R_\pi)$ is the projective spectrum of the pencil \[R_\pi(z)=z_0I+z_1\pi(a)+z_2\pi(b)+z_3\pi(c)+z_4\pi(d),\] where $z\in {\mathbb P}^4$. Proposition 3.7 implies that $p(R_\pi)$ is an invariant set for the map $G$ above. While a full description of the spectrum $p(R_\pi)$ or the Julia set ${\mathcal J}(G)$ seems hard to obtain at this point, the following conjecture seems natural in view of Theorem 5.11, Corollary 6.2 and Corollary 6.5.

\begin{Conj}
 $p(R_\pi)\subset {\mathcal J}(G)$.
\end{Conj}

\section{Concluding remarks}

Theorem \ref{thm: main} and Corollary 6.5 are clear evidences of a natural connection among self-similarity, Julia set, and projective spectrum. Even though $D_\infty$ is a rather simple group, the associated dynamical map $F$ defined in (\ref{eq: FP}) is by no means trivial. Two main factors that contributed to the establishment of Theorem \ref{thm: main} are:
 \begin{enumerate} \item an explicit description of the projective spectrum of $D_\infty$ (Theorem \ref{thm: PA});
 \item the semi-conjugacy of $F$ with the Tchebyshev polynomial $T(x)$ (Proposition 5.3). 
\end{enumerate}
 No analogous facts are currently known for the Grigorchuk group ${\mathcal G}$. However, despite the fact that ${\mathcal G}$ is infinitely presented, the $2$-similarity of its Koopman representation $\pi$ as described in Section 3.3 seems simple enough to warrant some further progresses along this line. It is therefore our great interest to continue this study for the group ${\mathcal G}$ and also to investigate on some other self-similar groups.

\section{Appendix}

\subsection{Proof of Proposition \ref{prop: Ffix}}

$(a)$ Finding fixed points is generally not difficult. For the map $F$ it involves solving the following system of equations:
\[ z_0(z_0^2-z_1^2-z_2^2)= z_0,\quad z_1^2 z_2=z_1, \quad  z_2(z_0^2-z_2^2)=z_2 .\]
By direct computation one verifies that only $(0,0,0), (0, \pm i, \mp i),$  and the points in the surface \[Y_F:=\{ z~\in~\mathbb{C}^3~\mid~z_0 ^2 -z_2 ^2 =1, z_1 =0 \}\] are fixed under $F$.
Observe that $\{(0,0,0), (0, \pm i, \mp i)\}\subset P(A_\pi)$  and $Y_F\subset P^c(A_\pi)$.
To classify the fixed points, one computes that
\[
F'(z) = \big( \frac{\partial f_{j-1}}{\partial z_{k-1}} \big)=
\begin{bmatrix}
3z_0 ^2 -z_1 ^2 -z_2 ^2 & -2z_0 z_1 & -2z_0 z_2 \\
0 & 2z_1 z_2 & z_1 ^2 \\
2z_0 z_2 & 0 & z_0 ^2 -3z_2 ^2
\end{bmatrix},
\] and hence $det( F'(z))= 6z_1z_2 (z_0^2-z_2^2)(z_0^2-z_1^2-z_2^2).$

$(b)$ To verify that the point $(0,0,0)$ is the only attracting fixed point of $F$, we first observe that it is a super attracting fixed point since
\[ 
\det \big| F'\big((0,0,0) \big) - \lambda I \big| = 
\det \begin{bmatrix}
-\lambda &0 & 0 \\
0 & -\lambda &0 \\
0 & 0 & -\lambda
\end{bmatrix}
=\lambda^3.
\]
Then we must then check that no points in $Y_F$ is attracting. If $z \in Y_F$ then
\begin{align*} 
\det \big| F'\big(z \big) - \lambda I \big| &= \det
\begin{bmatrix}
3z_0 ^2 - z_2 ^2 -\lambda &0 & -2z_0 z_2 \\
0 & -\lambda &0 \\
0 & 0 & z_0 ^2 -3z_2 ^2 - \lambda
\end{bmatrix} \\
&= -\lambda(2z_0^2 +1 -\lambda ) (1-2z_2 ^2 -\lambda)=0.
\end{align*}
Then $ \lambda =0$ is an eigenvalue, and the other two eigenvalues are $ \lambda_1=2z_0 ^2 +1$ and $\lambda_2 = 1-2z_2 ^2$. Since
\[ \lambda_1 + \lambda_2 = 2(z_0^2 -z_2 ^2 +1)=4, \]
either $|\lb_1|$ or $|\lb_2|$ must be larger than $1$, and hence $z$ is not attracting. Further, none of the points in $Y_F$ are repelling because $0$ is always an eigenvalue for $F'(z)$ as we just saw.
Having only one attracting fixed point is notable because it means that there is only one basin of attraction at a finite point. One may also verify directly that $(0,\pm i, \mp i)$ are repelling points of $F$.

$(c)$ To determine the fixed points for $F$ as defined in (\ref{eq: FP}) one observes that a point $z=[z_0:z_1z_2]$ is fixed if there exists a complex number $\lambda\neq 0$ such that
\[z_0(z_0^2-z_1^2-z_2^2) =\lambda z_0; \ z_1^2 z_2=\lambda z_1; \ z_2(z_0^2-z_2^2)=\lambda z_2.\]
Since the calculation is similar we shall not include it here.
\endproof

\subsection{Proof of Proposition \ref{prop: Gfix}}

To find the fixed points of $G$ we must solve the following system of equations:
\begin{equation}
\label{eq: grig eqs}
z_0 \al-z_1^2(z_0+z_4)=z_0,\   z_1^2 (z_2+z_3)=z_1, \  z_4 \al= z_2,\  z_2 \al=z_3,\   z_3 \al=z_4.
\end{equation}
We will find that the $\al$ notation is quite useful in this computation. We start by focusing on the final three conditions which have an apparent relationship, i.e.
\[z_4\al= z_2, \quad z_2 \al=z_3, \quad  z_3 \al=z_4.\]
Plugging in $z_4$ from the last equation into the first gives $z_3 \al^2 =z_2$. We then plug this $z_2$ into the middle equation to obtain $z_3 \al^3 = z_3.$
Therefore, if $G(z)=z$ then either $z_3=0 \ \text{ or } \ \al^3 =1.$
This gives us three cases to discuss separately. 

First, if $z_3=0$, then $z_2=z_4=0$, and hence the second equation implies $z_1=0$. This means the first expression becomes $z_0 \al = z_0.$
Then either $z_0 =0$ or $\al = 1$. But if $z_1=z_2=z_3=z_4=0$ then $\al = z_0^2.$
Therefore our first case yields the following fixed points
\begin{equation}
\label{eq: 3 fixed pts}
{\bf 0}=(0,0,0,0,0)  \text{   and   }  (\pm1, 0,0,0,0).
\end{equation}

In the second case, we assume that $\al =1$. Then $z_2=z_3=z_4=\gamma \neq 0$ and our remaining equations become
\[z_0-z_1^2(z_0+\gamma) = z_0 \quad \text{ and } \quad z_1^2 (2 \gamma)=z_1.\] 
We split this into two more subcases based on the second equation, namely either
\[ z_1=0 \quad \text{ or } \quad z_1 = \frac{1}{2\gamma}. \]
First, if $z_1=0$ then $z_0$ must still satisfy that
\[\al = (z_0+ \gamma)^2 - 4\gamma =1.\]
Hence
\[z_0 = -\gamma \pm \sqrt{1+4 \gamma^2}.\]
This gives a surface of fixed points
\begin{equation}
\label{eq: G fixed surface}
 Y_1:=\{z \in \C^5 \mid z= (-\gamma \pm \sqrt{1+4 \gamma^2},0, \gamma, \gamma, \gamma), \gamma \neq 0\}.
\end{equation}
Lastly, we examine the case when $z_1 = \frac{1}{2 \gamma}$. Then for the first component to be fixed we must have
\[ z_0- \frac{1}{4\gamma^2}(z_0+\gamma) = z_0.\]
This requires that $z_0 = -\gamma$, and hence $\al= (2\gamma)^2-(2\gamma)^2 =0 \neq 1.$
This contradicts the beginning of this subcase when $\al =1$, therefore this case yields no additional fixed points.

The third case to consider is when $\alpha^2+ \alpha+1=0$. We manipulate the last three equations of (\ref{eq: grig eqs}) to find $z_2$ and $z_3$ in terms of $z_4$ thus
\[ z_2 = z_4 \alpha \quad \text{ and } \quad z_3 = z_4 \alpha ^2.\]
Then the first two equations become
\[ z_0 (\alpha -1) -z_1 ^2(z_0+z_4)=0 \quad \text{ and } \quad -z_1(z_1 z_4+1)=0 .\]
Focusing on the later, we first examine if $z_1=0$ then
\[z_0 (\alpha -1) =0.\]
Clearly, $\alpha \neq 1$ in this case; then we must have $z_0=0$, but this contradicts our equation for $\alpha$
\[ \alpha = (z_0+z_4)^2-(z_2+z_3)^2= z_4^2 -z_4^2=0,\]
so there are no fixed points from this case when $z_1=0$. Then let $z_1= -\frac{1}{z_4}$ and $z_4 \neq 0$ because we addressed that case previously. Then for the point to be fixed
\[z_0(\alpha-1) -\frac{1}{z_4^2} (z_0+z_4) =0.\]
Thereby,
\[z_0 = \frac{z_4}{z_4^2(\alpha-1)-1}\]
the last step is to check what $z_4$ satisfies our equation for $\alpha$. 
\[\alpha = (z_0+z_4)^2- (z_2+z_3)^2 = z_4^2\bigg( \frac{1}{z_4^2(\alpha-1)-1}+1 \bigg)^2 -z_4^2 = z_4^2\bigg(\frac{2z_4^2(\alpha-1)-1}{(z_4^2(\alpha -1)-1)^2} \bigg).\]
Then we solve for $z_4$ while utilizing the fact that $\alpha^2 +\alpha +1 =0$ to obtain
\begin{align*}
 \big(\alpha (\alpha-1)^2-2(\alpha-1) \big)z_4^4 + \big(1-2 \alpha (\alpha -1) \big) z_4^2 + \alpha &=0 \\
(\alpha +5) z_4^4 + (4 \alpha +3) z_4^2 +\alpha &=0,
\end{align*}
and consequently
\[ z_4^2 = w_{\pm}(\alpha):=\frac{-4 \alpha -3 \pm i \sqrt{8\alpha +3}}{2 \alpha +10}.\]
Denoting $z_4$ by $\beta$, we have the set of fixed point $Y_2$ as describe before Proposition \ref{prop: Gfix}. 

To classify the fixed points of $G$ using the Jacobian, we use the notation $\beta = z_0+z_4$ and $\eta = z_2+z_3$ for convenience and compute that
\begin{align*}
G'(z)=
\begin{bmatrix}
\al +2z_0\beta-z_1^2  & 2z_1 \beta & -2z_0 \eta & -2z_0 \eta^2 & 2 z_0 \beta-z_1^2 \\
0 & 2z_1 \eta & z_1^2 & z_1^2 & 0 \\
2z_4 \beta & 0 & -2z_4 \eta & -2z_4 \eta & \al +2z_4\beta\\
2z_2 \beta & 0 & \al -2z_2 \eta & -2z_2 \eta & 2z_2\beta \\
2z_3 \beta & 0& -2z_3 \eta & \al -2z_3 \eta & 2z_3 \beta 
\end{bmatrix}.
\end{align*} 

To determine the fixed points for $G: {\mathbb P}^4\to  {\mathbb P}^4$, one considers the system of equations:
\begin{equation}
\label{eq: grig eqs}
z_0 \al-z_1^2(z_0+z_4)=\lb z_0,\   z_1^2 (z_2+z_3)=\lb z_1, \  z_4 \al= \lb z_2,\  z_2 \al=\lb z_3,\   z_3 \al=\lb z_4.
\end{equation}
Going through similar steps one obtains the set of fixed point $Y_1$ and $Y_2$. However, an additional point $[-1: -2: 1: 1: 1]$ is obtained in the case $z_2=z_3=z_4$ and $\alpha(z)=\lb$. One checks that $G(-1,-1,1,1,1)=(4,8,-4,-4,-4)$. One can verify using the commutative diagram (\ref{eq: GM}) that $[-1: -2: 1: 1: 1]\in p(R_\pi)$. However, determining which points in $Y_2$ are in $p(R_\pi)$
does not seem easy.

To see that the origin ${\bf 0}$ is the only attracting fixed point of $G$, one first observes that since
\begin{align*} 
\det \big| G'({\bf 0}) - \lambda I \big| &= 
\det \begin{bmatrix}
-\lambda &0 & 0 &0&0\\
0 & -\lambda &0 &0&0\\
0 & 0 & -\lambda &0&0 \\
0&0 & 0 & -\lambda &0 \\
0&0 & 0 &0 & -\lambda 
\end{bmatrix}=\lambda^5,
\end{align*}
the origin ${\bf 0}$ is indeed an attracting fixed point.
Now assume $z \in Y_1$. For simplicity we won't substitute in for $z_0$ immediately. Using the fact that in this case $\al(z)=1$, we compute $\det \big| G'(z) - \lambda I \big|$ as
\begin{align*} & 
\det \begin{bmatrix}
1 +2z_0(z_0+\gamma) -\lb &0 & -4z_0 \gamma &-8 z_0 \gamma^2&2z_0(z_0+ \gamma)\\
0 & -\lambda &0 &0&0\\
2\gamma(z_0+\gamma) & 0 &-4 \gamma^2 -\lambda &-4\gamma^2& 1 +2 \gamma (z_0+\gamma) \\
2\gamma(z_0+\gamma)&0 &1-4\gamma^2  &-4\gamma^2 -\lambda &2\gamma(z_0+\gamma) \\
2\gamma(z_0+\gamma)&0 & -4\gamma^2 &1-4\gamma^2 &2\gamma(z_0+\gamma) -\lambda 
\end{bmatrix}\\
&= 
 -\lb\det \begin{bmatrix}
1 +2z_0(z_0+\gamma)-\lb& -4z_0 \gamma &-8 z_0 \gamma^2&2z_0(z_0+ \gamma)\\
2\gamma(z_0+\gamma)  &-4 \gamma^2 -\lambda &-4\gamma^2& 1 +2 \gamma (z_0+\gamma) \\
2\gamma(z_0+\gamma) &1-4\gamma^2  &-4\gamma^2 -\lambda &2\gamma(z_0+\gamma) \\
2\gamma(z_0+\gamma) & -4\gamma^2 &1-4\gamma^2 &2\gamma(z_0+\gamma) -\lambda 
\end{bmatrix}\\
&= 
 -\lb\det \begin{bmatrix}
1 +2z_0(z_0+\gamma)-\lb& -4z_0 \gamma &-8 z_0 \gamma^2&2z_0(z_0+ \gamma)\\
2\gamma(z_0+\gamma)  &-4 \gamma^2 -\lambda &-4\gamma^2& 1 +2 \gamma (z_0+\gamma) \\
0&1+ \lb & -\lambda &-1\\
0 & \lb&1 &-1 -\lambda 
\end{bmatrix}.
\end{align*}
Then expanding using the third row, we can break the determinant above into a sum of three parts, i.e., $\det \big| G'(p) - \lambda I \big|  = A+B+C$ where
\begin{align*}
A &= \lb(1+\lb)\det \begin{bmatrix}
1 +2z_0(z_0+\gamma)-\lb &-8 z_0 \gamma^2&2z_0(z_0+ \gamma)\\
2\gamma(z_0+\gamma)  &-4\gamma^2& 1 +2 \gamma (z_0+\gamma) \\
0 &1 &-1 -\lambda 
\end{bmatrix},\\
B & =
\lb^2 \det \begin{bmatrix}
1 +2z_0(z_0+\gamma)-\lb& -4z_0 \gamma &2z_0(z_0+ \gamma)\\
2\gamma(z_0+\gamma)  &-4 \gamma^2 -\lambda & 1 +2 \gamma (z_0+\gamma) \\
0 & \lb&-1 -\lambda 
\end{bmatrix}, \\
C &=-\lb \det \begin{bmatrix}
1 +2z_0(z_0+\gamma)-\lb& -4z_0 \gamma &-8 z_0 \gamma^2\\
2\gamma(z_0+\gamma)  &-4 \gamma^2 -\lambda &-4\gamma^2 \\
0 & \lb&1
\end{bmatrix}.
\end{align*}
After simplicafication we obtain the eigenvalue equation
\[\det (G'(z)-\lambda I)= \lb (\lb^2 + \lb +1) \beta =0, \]where
 \[\beta=  (1+2z_0(z_0+\gamma)-\lb)(6 \gamma^2 -2 \gamma z_0 + \lb -1)-8 \gamma^2z_0(z_0+\gamma) (2\gamma+1)+4z_0 \gamma(z_0+\gamma)^2.\]
This shows that three of the eigenvalues are $\{0, \frac{-1\pm \sqrt{3}i}{2} \}$. Thus, there exist $\lb_1, \lb_2$ such that $|\lb_1|=|\lb_2|=1$, hence the point $z$ is neither attracting nor repelling.

Likewise we can show by direct computation using Matlab that the eight points in $Y_2$ are not attracting, because the Jacobian $G'$ at each point in the set has at least one eigenvalue with modulus greater than $1$. We omit the computations here for simplicity.


\begin{thebibliography}{1}
\bibitem{BG} L. Bartholdi, R. Grigorchuk, {\em On the Spectrum of Hecke Type Operators Related to some Fractal Groups}, Tr. Mat. Inst. im. V.A. Steklova, Ross. Akad. Nauk 231 (2000), 5-45. Proc. Steklov Inst. Math. 231 (2000), 1-41.

\bibitem{BCY} J. Bannon, P. Cade and R. Yang, {\em On the Spectrum of Operator-valued Entire Functions}, Illinois J. of Mathematics 55 No.4 (2011).

\bibitem{Be} A. F. Beardon, {\em Iteration of Rational Functions}, Graduate Text in Mathematics 132, Springer-Verlag, New York, 1991.

\bibitem{BHV} B. Bekka, de la Harpe and A. Valette, {\em Kazhdan's Property (T)}, New Mathematical Monographs 11, Cambridge University Press, Cambridge, 2008.

\bibitem{CY} P. Cade and R. Yang, {\em Projective Spectrum and Cyclic Cohomology}, J. of Funct. Analy. Vol. 265 No. 9 (2013), page 1916–1933.

\bibitem{CSZ} I. Chagouel, M. Stessin and K. Zhu, {\em Geometric Spectral Theory for Compact Operators}, Trans. Amer. Soc. 368 (2016), No. 3, 1559-1582.

\bibitem{Cu} C. Curtis, {\em Representation Theory of Finite Groups: from Frobenius to Brauer}, Math. Intelligencer 14 (1992), 48-57.

\bibitem{Cu2} C. Curtis, {\em Pioneers of Representation Theory: Frobenius, Burnside, Schur, and Brauer}, History of Mathematics, Providence, R.I.: American Mathematical Society, 2003.

\bibitem{De} R. Dedekind, {\em Gesammelte Mathematische Werke}, Vol. II. Chelsea, New York, 1969.

\bibitem{Di} J. Dixmier, {\em Les $C^*$-alg\`{e}bres et leurs repr\'{e}sentations}, Gauthier-Villars, 1969.

\bibitem{Di21} L. E. Dickson, {\em Determination of all General Homogeneous Polynomials Expressible
as Determinants with Linear Elements}, Trans. Amer. Math. Soc. 22 (1921), no. 2,167 - 179.

\bibitem{Di75} L. E. Dickson, {\em An Elementary Exposition of Frobenius Theory of Group Characters and Group Determinants}, Ann. of Math. 4 (1902), 25-49; Mathematical Papers, Vol. II. Chelsea, New York,1975, 737-761.

\bibitem{DY} R. G. Douglas and R. Yang, {\em Hermitian Geometry on Resolvent set (I)}, Operator theory, operator algebras, and matrix theory, 167–183, Oper. Theory Adv. Appl., 267, Birkh\"{a}user/Springer, Cham, 2018.

\bibitem{DG} A. Dudko and R. Grigorchuk, {\em On Spectra of Koopman, Groupoid and Quasi-regular Representations}, arXiv: 1510.00897v3. 

\bibitem{FS} E. Formanek and D. Sibley, {\em The Group Determinant Determines the Group}, Proc. A.M.S Vol 112, No. 3 (1991), 649-656.

\bibitem{Forn} J. Fornaess, {\em Dynamics in Several Complex Variables}, CBMS Regional Conference Series in Mathematics, 87. Published for the Conference Board of the Mathematical Sciences, Washington, DC; by the American Mathematical Society, Providence, RI, 1996. viii+59 pp.

\bibitem{For} J. Fornaess, {\em The Julia set of H\'{e}non maps}, Mathematische Annalen 334 (2006), Issue 2, pp 457-464.

\bibitem{FS2} J. Fornaess and N. Sibony, {\em Complex Dynamics in Higher Dimension II}, Annals of Math. Stud. 1995, 135-182.

\bibitem{Fr} F. G. Frobenius, {\em \"{U}ber vertauschbare Matrizen, Sitzungsberichte der K\"{o}niglich Preussischen}, Akademie der Wissenschaften zu Berlin (1896) 601-614; Gesammelte Abhandlungen, Band II Springer-Verlag, New York, 1968, 705-718.



\bibitem{GOY} B. Goldberg, R. Yang, {\em Hermitian Metric and the Infinite Dihedral Group }, Proc. of the Steklov Institute of Math.,2019, Vol. 304, pages 149-158.

\bibitem{Gr80} R. Grigorchuk, { \em On Burnside's Problem on Periodic Groups}, Funktsional. Anal. i Prilozhen. 14 (1980), no. 1, 53-54, english translation: Functional Anal. Appl. 14 (1980), 41-43.

\bibitem{Gr84} R. Grigorchuk, {\em Degrees of Growth of Finitely Generated Groups and the Theory
of Invariant Means}, Izv. Akad. Nauk SSSR Ser. Mat. 48 (1984), no. 5, 939-985,
english translation: Math. USSR-Izv. 25 (1985), no. 2, 259-300.

\bibitem{Gr85} R. Grigorchuk, {\em Degrees of Growth of p-Groups and Torsion-free Groups}, Mat.
Sb. (N.S.) 126(168) (1985), no. 2, 194-214, 286.

\bibitem{GLS} R. Grigorchuk, D. Lenz and T. Smirnova-Nagnibeda, {\em Spectra of Schreier graphs of Grigorchuk's group and Schroedinger Operators with Aperiodic Order}, arXiv:1412.6822.

\bibitem{GN07} R. Grigorchuk and V. Nekrashevych, {\em Self-similar Groups, Operator Algebras and Schur Complement}, J. Mod. Dyn. 1 (2007), no. 3, 323-370.

\bibitem{GNS} R. Grigorchuk, V. Nekrashevich and V. Sushchanski\v{i}, {\em  Automata, Dynamical Systems, and Groups}, (Russian) Tr. Mat. Inst. Steklova 231 (2000), Din. Sist., Avtom. i Beskon. Gruppy, 134--214; translation in Proc. Steklov Inst. Math. 2000, no. 4 (231), 128–203.

\bibitem{GNS2} R. Grigorchuk, V. Nekrashevich, and Z. Sunic, {\em From Self-similar Groups to Self- similar Sets and Spectra}, From Self-Similar Groups to Self-Similar Sets and Spectra. In: Bandt C., Falconer K., Z\"{a}hle M. (eds), Fractal Geometry and Stochastics V. Progress in Probability, vol 70. Birkh\"{a}user, Cham, 2015.

\bibitem{GS} R. Grigorchuk and Z. \v{S}uni\'{c}, {\em Schreier Spectrum of the Hanoi Towers Group on Three Pegs}, Proc. of Symposia in Pure Math. Vol. 77, 2008.


\bibitem{GY} R. Grigorchuk and R. Yang, {\em Joint Spectrum and the Infinite Dihedral Group}, Proc. of the Steklov Institute of Math., 2017, Vol. 297, 145-178.

\bibitem{Hal} P. R. Halmos, {\em Two Subspaces}, Trans. Amer. Math. Soc. 144 (1969), 381-389.



\bibitem{He} M. H\'{e}non, {\em A Two-dimensional Mapping with a Strange Attractor}, Communications in Mathematical Physics. 50 (1976) no.1, 69-77.

\bibitem{HWY} W. He, X. Wang and R. Yang, {\em Projective Spectrum and Kernel Bundle(II)}, J. of Operator Theory 78 (2017), no. 2, 417-433.

\bibitem{HY} W. He and R. Yang, {\em Projective Spectrum and Kernel Bundle}, Sci. China. Math. Vol. 57 (2014), 1-10.

\bibitem{HY18} Z. Hu and R. Yang, {\em On the Characteristic Polynomials of Multiparameter Pencils}, Linear Alg. and its Appli. 558 (2018), 250-263.


\bibitem{La} S. Latt\`{e}s, {\em Sur l'iteration des Substitutions Rationelles et les Fonctions de Poincar\'{e}}, Comptes Rendus Acad. Sci. Paris 166 (1918), 26-28.

\bibitem{LM} M. Lyubich and M. Martens, {\em Renormalization of H\'{e}non maps}, Dynamics, Games and Science I (2011), Peixoto M., Pinto A., Rand D. (eds), Springer Proceedings in Mathematics, vol 1, Springer, Berlin, Heidelberg, 597-618.

\bibitem{MQW17} T. Mao, Y. Qiao and P. Wang, {\em Commutativity of Normal Compact Operator via Projective Spectrum}, Proc. A.M.S 146 (2017), no. 3, 1165-1172.

\bibitem{MH} J. Mason and D. Handscomb, {\em Chebyshev Polynomials}, Chapman \& Hall/CRC, Boca Raton, FL 2003

\bibitem{JM} J. Milnor, {\em Dynamics in One Complex Variable}, Princeton University Press, Princeton, NJ 2006.

\bibitem{MNTU}S. Morosawa,Y. Nishimura, M. Taniguchi, and T. Ueda, {\em Holomorphic Dynamics}, Volume 66, Cambridge Studies in Advanced Mathematics. Cambridge University Press, Cambridge, 2000.

\bibitem{Ne05} V. Nekrashevych, {\em Self-similar Groups}, Mathematical Survey and Monographs, A.M.S Providence, RI, 2005.

\bibitem{Ne18} V. Nekrashevych, {\em Palindromic Subshifts and Simple Periodic Groups of Intermediate Growth}, Annals Math. 187 (2018), Issue 3, 667-719.

\bibitem{RS} I. Raeburn and A. Sinclair, {\em The $C^*$-algebras Generated by two Projections}, Math. Scand. 65 (1989), 278-290.

\bibitem{ST} M. Stessin and A. Tchernev, {\em Spectral Algebraic Curves and Decomposable Operator Tuples}, arxive: 1509.06274v1.

\bibitem{SYZ} M. Stessin, R. Yang and K. Zhu, {\em Analyticity of a Joint Spectrum and a Multivariable Analytic Fredholm Theorem}, New York J. Math. 17A (2011), 39-44.

\bibitem{Ue} T. Ueda, {\em Fatou Sets in Complex Dynamics on Projective Spaces}, J. Math. Soc. Japan  46 (1994), No. 3, 545-555.

\bibitem{Ya} R. Yang, {\em Projective Spectrum in Banach Algebras}, J. Topol. and Analy. 1 (2009), No. 3, 289-306.



\end{thebibliography}
\end{document}